\documentclass[11pt]{amsart}

\usepackage[colorlinks=true,citecolor=black!60!green,linkcolor=black!60!red,filecolor=black!60!cyan,urlcolor=black!60!magenta]{hyperref}
\usepackage[text={6.1in,8.5in},centering]{geometry}
\usepackage{amssymb,amsmath,amsthm,combelow,mathtools,extpfeil,wrapfig,graphicx}
\usepackage{caption}
\usepackage{subcaption}
\input{insbox}
\usepackage[all,cmtip]{xy}
\usepackage{tikz}
\usepackage[shortlabels]{enumitem}

\newtheorem{thm}{Theorem}
\newtheorem*{thm*}{Theorem}
\newtheorem*{lem*}{Lemma}
\newtheorem{thmA}{Theorem}

\newtheorem{lem}[thm]{Lemma}
\newtheorem{prop}[thm]{Proposition}
\newtheorem*{prop*}{Proposition}
\newtheorem{cor}[thm]{Corollary}

\theoremstyle{definition}

\newtheorem{rmk}[thm]{Remark}
\newtheorem*{rmks}{Remarks}
\newtheorem{question}[thm]{Question}
\newtheorem{ex}[thm]{Example}

\numberwithin{thm}{section}
\numberwithin{equation}{section}

\DeclareMathOperator{\id}{id}

\DeclareMathOperator{\vol}{vol}

\DeclareMathOperator{\Dil}{Dil}

\DeclareMathOperator{\Map}{Map}
\DeclareMathOperator{\Lip}{Lip}
\DeclareMathOperator{\Hom}{Hom}

\newcommand{\ph}{\varphi}
\newcommand{\epsi}{\varepsilon}

\newcommand*\strong[1]{\textbf{\textit{#1}}}

\newcommand{\CC}{\mathbb{C}}

\newcommand{\QQ}{\mathbb{Q}}
\newcommand{\ZZ}{\mathbb{Z}}
\newcommand{\NN}{\mathbb{N}}
\newcommand{\RR}{\mathbb{R}}

\begin{document}
\title{Persistent homology of function spaces}
\author[J.~Block]{Jonathan Block}
\address[J.~Block]{Department of Mathematics, University of Pennsylvania, Pennsylvania, US}
\email{blockj@math.upenn.edu}
\author[F.~Manin]{Fedor Manin}
\address[F.~Manin]{Department of Mathematics, University of Toronto, Ontario, Canada}
\email{manin@math.ucsb.edu}
\author[Sh.~Weinberger]{Shmuel Weinberger}
\address[Sh.~Weinberger]{Department of Mathematics, University of Chicago, Illinois, US}
\email{shmuel@math.uchicago.edu}

\begin{abstract}
  We can view the Lipschitz constant as a height function on the space of maps
  between two manifolds and ask (as Gromov did nearly 30 years ago) what its
  ``Morse landscape'' looks like: are there high peaks, deep valleys and
  mountain passes?  A simple and relatively well-studied version of this
  question: given two points in the same component (homotopic maps), does a
  path between them (a homotopy) have to pass through maps of much higher
  Lipschitz constant?  Now we also consider similar questions for
  higher-dimensional cycles in the space.  We make this precise using the
  language of persistent homology and give some first results.
\end{abstract}
\maketitle

\section{Introduction}

\subsection{Background}
The purpose of this paper is to give some information about the Morse landscape of function spaces, with the Lipschitz constant acting as a height function.

Suppose $X$ and $Y$ are finite complexes, and we give them ``reasonable'' metrics (e.g.\ path metrics compatible with a linear structure on simplices); then all continuous functions can be approximated by Lipschitz ones, and indeed the inclusion of the space $\Lip(X,Y)$ of Lipschitz functions in the continuous ones $\mathcal C(X,Y)$ is a weak homotopy equivalence.  However, Lipschitz functions have a natural notion of complexity, namely the Lipschitz constant: for example, the space of functions from $X$ to $Y$ with Lipschitz constant $\leq L$ is compact.  Moreover, if $Y$ is a locally CAT($K$) space for some finite $K$ (any compact Riemannian manifold fits the bill), then the image of this space in the functions with constant $\leq L+\epsi$ has finitely generated homology (Theorem \ref{tameness}).

This suggests a program of trying to understand: at what level are particular homology classes in $H_i(\Lip(X,Y))$ born?  More generally, we can try to understand how the topology of the spaces of $L$-Lipschitz maps evolves as $L$ increases.

Persistent homology (see e.g.~\cite{ELZ} for the paper that first used this term and \cite{PRSZ} for a monograph on some of its applications in pure mathematics) provides a natural language, called ``barcodes'' \cite{ZC},  for describing this structure---at least with field coefficients\footnote{Although our results will be phrased in terms of persistent homology, the obvious rephrasing in terms of homology and induced maps applies in integral homology.}.  For a Morse function filtering a manifold by sublevel sets, each critical value gives rise to either the beginning or end of a ``bar'', indicating either the birth or death of a homology class.

The set of barcodes admits a natural metric (based on throwing away some bars of length $<2\epsi$ and aligning the remaining bars up to $\epsi$-errors), and $C^0$-close functions have close barcodes \cite{CEH}.  Since the different linear metrics on a finite complex $Y$ are Lipschitz-related (i.e.~the identity map is a bilipschitz equivalence), the function $\log\Lip$ on $\Lip(X,Y)$ changes by a uniformly bounded amount when one changes the metric---which means that the persistence barcodes of function spaces are (up to finite distance) topologically invariant; indeed, they turn out to be homotopy invariant (Theorem \ref{thm:interleaving}).  Of course the infinite bars are just the ordinary homology of the function space, but some extra information is given by where they start (when the total homology of the function space is infinitely generated), and the information about finite length bars is entirely new.

These questions were first raised in a vague form by Gromov in \cite{GrQHT} and \cite[7.20]{GrMS}, although he made precise the question about $PH_0$: how much must one increase the Lipschitz constant to obtain a homotopy between homotopic $L$-Lipschitz maps?  We study this problem below for all $PH_i$, though some of our results are new even for $PH_0$.  Other papers that deal with these questions are \cite{GrHED,FWPNAS,NabRot,CDMW,CMW,PCDF,scal,BGM}.

\subsection{Main results}
We write $PH_*(A,f)$ for the persistent homology of a space $A$ equipped with a filtration by sublevel sets of a function $f:A \to \RR$.  We leave coefficients implicit, although they must be a field for ``bars'' to be taken literally.

The following theorem is a quite elementary consequence of convexity properties of distance functions on nonpositively curved spaces, and covering space theory.
\begin{thmA} \label{main:CAT0}
  If $Y$ is a complete locally CAT(0) metric space (e.g.~a complete
  nonpositively curved manifold) then all of the bars in
  $PH_*(\Lip(X,Y), \Lip)$ are of infinite length.
\end{thmA}

Note that this result is highly metric-dependent: if $Y$ is a compact hyperbolic manifold, one can easily give $Y$ a metric where, say, for $X=S^1$, there are infinitely many finite-length bars; indeed the number (measured from where their bottoms lie) grows exponentially with the Lipschitz constant $L$.  If $X$ is higher-dimensional, the number of such bars will typically be $\exp(O(L^{\dim(X)}))$ (as is implicit in \cite[p.~36]{GrMS}).  However, when $Y$ is compact, for all Riemannian metrics on $Y$ (and indeed on compact manifolds homotopy equivalent to $Y$), all the bars are of uniformly bounded size with respect to $\log\Lip$.

For general $Y$, there may be infinitely many bars of arbitrarily long finite length.  For example, this is necessarily the case when $X$ is $S^1$ and $Y$ has fundamental group with unsolvable word problem (or even superexponential Dehn function \cite{WeiWhat,WeiCRM}).  In order to avoid the difficulties caused by the fundamental group and concentrate on the impact of homotopy theory when studying our function spaces, we will assume now that $Y$ is simply connected.

In this case, using methods from \cite{GrMS} and the remarkable paper \cite{NabRot}, we get, for maps of a circle into $Y$:
\begin{thmA} \label{main:NR}
  Let $Y$ be a simply connected finite complex and let $\Omega Y$ and
  $\Lambda Y$ denote the spaces of based and free Lipschitz loops,
  respectively.  Then all of the finite length bars in $PH_i(\Omega Y, \Lip)$
  and $PH_i(\Lambda Y, \Lip)$ are uniformly bounded in length; the bound is
  linear in $i$.
\end{thmA}
Gromov, in \cite[\S4]{GrHED} and \cite[Theorem~7.3]{GrMS}, shows that for both free and based loop spaces the bottoms of the infinite bars in $PH_i$ lie within a range $C_1i \leq \Lip \leq C_2i$; this suggests that for the finite bars, too, one may not be able to do better than a linear bound.

For the more general simply connected case, where the domain is not necessarily the circle, we have available the tools of rational homotopy theory \cite{Qui,SulLong,GrMo}.  Indeed, one could well be guided towards results of our general sort by Quillen's algebraicization of rational homotopy theory via Lie algebras, which suggests a connection to the deep theory of filling functions on nilpotent Lie groups.   However, we will largely make use of the Sullivan theory, which uses differential forms, and the connection between these and the Lipschitz functional given by the ``shadowing principle'' of \cite{PCDF}.

Philosophically, loop spaces are an unusually simple (though important) case because the cohomology of a loop space is a Hopf algebra, so one knows a lot a priori about the infinite length bars.  However, we have not been able to exploit this to give nearly as strong information as Theorem \ref{main:NR}.  Indeed, Theorems \ref{main:CAT0} and \ref{main:NR} are the only two we are aware of where one can get information about $\Lip$, as opposed to $\log \Lip$.

When the domain is higher-dimensional we find that uniform boundedness of persistence intervals depends strongly on the rational homotopy type of $Y$.  We emphasize, though, that we are not restricting ourselves to rational coefficients and that the persistent homology groups in the theorems below are not rationally invariant.
\begin{thmA} \label{main:pw}
  If $Y$ is a simply connected finite complex with positive weights (e.g.~$Y$
  is formal or a homogeneous manifold), then for all finite $X$, all of the
  finite length bars in $PH_i(\Lip(X,Y)_0, \log_+ \Lip)$ are of uniformly
  bounded length.  Here $\Lip(X,Y)_0$ is the connected component of the
  function space that includes the constant maps.
\end{thmA}

\begin{thmA} \label{main:H}
  If $Y$ is a simply connected finite complex rationally equivalent to an
  H-space (e.g.~an odd-dimensional sphere or a Lie group), then for all finite
  $X$, all of the finite length bars in $PH_i(\Lip(X,Y), \log_+ \Lip)$ are of
  uniformly bounded length.
\end{thmA}
Here $\log_+(x)=\max\{\log x,0\}$; we use $\log_+\Lip$ because we do not wish to emphasize the behavior of functions with very small Lipschitz constant.

In these theorems, as in Theorem \ref{main:NR}, we do not have uniformity in $i$; the statements are about uniformity for a fixed $i$ in the persistence parameter $\log \Lip$.  Moreover, outside the case of loop spaces we do not know how to study the question of growth with respect to $i$.

We show that at least some of these hypotheses are necessary, and that finite bars of unbounded length can exist for simply connected targets:
\begin{thmA} \label{main:ex}
  For arbitrarily large $L$, there are homotopic pairs of $L$-Lipschitz maps
  $f_L$ and $g_L: \CC P^2 \times S^3 \to S^3 \vee S^3$ for which any homotopy
  must go through maps with Lipschitz constant $>cL^{4/3}$, giving bars of
  linearly growing length in
  $PH_0(\Lip(\CC P^2 \times S^3,S^3 \vee S^3), \log_+ \Lip)$.
\end{thmA}

\begin{rmk}
  All of the results in this paper are true for ``persistent homotopy'' with
  identical proofs.  We preferred to express our results in terms of homology
  because of its relationship to Morse theory.
\end{rmk}

\subsection{Open questions}
This paper is the beginning of a theory and there are many gaps left to fill in.  We highlight some of these here:
\begin{question}
  Is the assumption of positive weights necessary in Theorem \ref{main:pw}?
\end{question}
\noindent While spaces without positive weights are relatively complicated, the proof of Theorem \ref{main:ex} may provide a guide to what to look for in a potential counterexample.
\begin{question}
  Can Theorem \ref{main:H} be extended to ``two-stage'' spaces (spaces that are
  rationally fibers of maps between products of Eilenberg--MacLane spaces, such
  as homogeneous manifolds)?
\end{question}
\noindent We show in Theorem \ref{thm:2step} that this is true for $PH_0$ under the additional assumption that $Y$ is \emph{scalable} in the sense of \cite{scal} (e.g.~a symmetric space, but not every homogeneous manifold), but it is unclear how to extend the result to higher $PH_i$.  On the other hand, the proof of Theorem \ref{main:ex} uses the presence of three ``stages'' in an essential way.

Rational homotopy theory works not only for simply connected spaces, but more generally for \emph{nilpotent spaces}: those which have a nilpotent fundamental group which acts nilpotently on all higher homotopy groups.  The shadowing principle of \cite{PCDF} is extended to nilpotent spaces in Kyle Hansen's thesis \cite{Kyle}; with this, the proof of Theorem \ref{main:pw} extends verbatim to nilpotent spaces with positive weights, including for example nilmanifolds whose fundamental group admits a (not necessarily Carnot) grading.  On the other hand, the following question is wide open:
\begin{question}
  Can Theorem \ref{main:NR} be extended to nilmanifolds (at least with respect
  to $\log\Lip$)?
\end{question}
\noindent A positive answer for $PH_0$ is due to Riley \cite{Riley}, see also \cite[II, Theorem~4.1.1]{RileyCourse}.

Similarly, one may ask whether an example similar to that of Theorem \ref{main:ex} can be found for a nilmanifold target, placing it within the domain of geometric group theory.

\subsection{Outline of the paper}
We start in \S\ref{S2} by explaining aspects of persistent homology relevant to our setting.  The rest of the paper, however, can be understood without most of the technical material there.  We next prove Theorem \ref{main:CAT0} in \S\ref{S3} and two proofs of variants of Theorem \ref{main:NR} in \S\ref{S4} and \S\ref{S:loops-Gro}.  In \S\ref{S5} we introduce prior results from quantitative rational homotopy theory which are used in the remainder of the paper.  In \S\ref{S6} we use these to prove Theorems \ref{main:pw}, and \ref{main:H}, as well as a more general, but weaker result in the spirit of Theorem \ref{main:NR}.  In \S\ref{S:example} we prove Theorem \ref{main:ex}, and in \S\ref{S:optimal} we give some more results about $PH_0$ which show that this example is in some sense both as simple as it can be and sharp.

\subsection*{Acknowledgements}
This paper is an outgrowth of many years of thinking about the problems posed by Gromov \cite{GrHED,GrQHT,GrMS}, and we must thank him most of all.  We also thank Alex Nabutovsky for pointing out the remarkable \cite{NabRot} when we had only proved Theorem \ref{thm:ph-loops}.  Theorem \ref{main:NR} answers a question asked by S.~W.\ in a talk \cite{WeiTalk} about a speculative model of evolution as a modified stochastic gradient flow on loop spaces, and he would like to thank Tuca Auffinger for discussions of this model, and the National Institute for Theory and Mathematics in Biology for providing the context for such speculations.  We also thank an anonymous referee for useful comments that helped improve the exposition.

F.~M.\ would like to thank the Hausdorff Institute\footnote{Funded by the Deutsche Forschungsgemeinschaft (DFG, German Research Foundation) under Germany's Excellence Strategy -- EXC-2047/1 -- 390685813.} for providing a beautiful space for proving theorems.  The authors were partially funded by an NSERC Discovery Grant and NSF individual grants 2204001, 2105451 and 2505738.

\section{Definitions and general facts} \label{S2}

We will be studying the filtration given by the sublevel sets of the Lipschitz constant on the set of Lipschitz functions $\Lip(X,Y)$ between two metric spaces $X$ and $Y$.  We denote the subset of $L$-Lipschitz functions by $\Lip_L(X,Y)$.  In this section we introduce the language and machinery of persistent homology and explain how it applies in this setting.  Although this is necessary to justify using the language of barcodes in the rest of the paper, the technical details are largely irrelevant to the proofs of our main results and the reader should feel free to skip them.

\subsection{Persistent homology}
Let $A$ be a topological space, and fix a filtration $\{A_t : t \in \RR\}$ of $A$ (i.e.~$A_t \subseteq A_s \subseteq A$ for $t \leq s$); we can think of this as a functor from the order category $(\RR, \leq)$ to the category of topological spaces.  We would like to study how the topology of the sets in the filtration evolves over the time parameter.  Intuitively, topological features may ``persist'' over various amounts of time; transient features are thought of as accidental and can be discounted, while long-lived ones tell us something meaningful about the nature of the filtration.

This frame of reference originates in topological data analysis, where one attempts to understand the underlying topology implied by a discrete data set; different times correspond to different scales at which the data set can be viewed, and a topological feature visible at a large range of scales is hypothesized to be present in the underlying distribution rather than an artifact of the data.  In the main example we will be studying, $A=\Lip(X,Y)$ and $A_t=\Lip_t(X,Y)$.  In this case, persistent topological features correspond to families of $L$-Lipschitz maps which cannot be contracted through maps of Lipschitz constant close to $L$.

Purely formally, the functor $H_n({-};R)$ induces a family of groups $\{H_n(A_t;R)\}$ and maps
\[i_*(t,s):H_n(A_t;R) \to H_n(A_s;R), \qquad \text{for }t \leq s,\]
i.e.~a functor from $(\RR,\leq)$ to $R$-modules, which contain all the homological information about the filtration.  Such a functor is called a \emph{persistence module}, and the persistence module consisting of the homology of $A_t$ is called the \emph{persistent homology} of $A$, notated $PH_n(A;R)$.  Here the filtration is implicit.  Our filtrations will typically consist of sublevel sets of continuous functions $f:A \to \RR$; in this case, we will use the notation $PH_n(A,f;R)$.

So far this invariant is essentially tautological; to make it meaningful one needs to analyze the structure of persistence modules $\mathbf M=(\{M_t\},\{\mu(t,s)\})$.  Such an analysis is readily available under the following simplifying assumptions, which can be summarized as saying that $\mathbf M$ has \emph{finite rank}:
\begin{enumerate}[(i)]
\item $R$ is a field.
\item Each $M_t$ is finite-dimensional.
\item The set $\{t_i\}$ of ``critical times'' for the filtration (i.e.~those
  $t_i$ for which for every $\epsi>0$ the induced map
  $\mu(t_i-\epsi,t_i+\epsi)$ is not an isomorphism) is either finite or has
  the order type of $\NN$.
\end{enumerate}
These assumptions are satisfied, for example, for $H_n(A,f;R)$ when $A$ is a manifold and $f$ is a proper Morse function.  In this case, one equips $\bigoplus_i H_n(A_{t_i};R)$ with the operation
\[\mathbf{t} \cdot h=i_*(t_i,t_{i+1})(h), \qquad h \in H_n(A_{t_i};R),\]
making it into an $R[\mathbf{t}]$-module.  By the structure theorem for modules over a PID, such a module is isomorphic to a unique direct sum of modules of the form $R[\mathbf{t}]$ and $R[\mathbf{t}]/(\mathbf{t}^j)$ (\emph{interval} or \emph{cyclic modules}).  These correspond to homology classes which are ``born'' at some time $t_i$ (which may be $-\infty$) and ``die'' at time $t_{i+j}$, or live forever in the case of modules $R[\mathbf{t}]$.

The data of such a persistence module can also be expressed by a \emph{persistence diagram} or \emph{barcode}, a multisubset of $[-\infty,\infty) \times (-\infty,\infty]$ consisting of pairs $(t_i,t_j)$ expressing the birth and death times---the lifetime interval---of the cyclic summands.  These intervals are referred to as \emph{bars}.

\begin{ex}
  Consider a Morse function $f:\RR^2 \to \RR$ with two local minima and one
  saddle point.  Intuitively, as $t$ increases, the sublevel sets
  $f^{-1}(-\infty,t]$ have two components that are ``born'' and then ``merge''.
  The two cyclic submodules in $PH_0(\RR^2,f)$ are one corresponding to the
  first component (an infinite bar) and one corresponding to the difference
  between the components (a finite bar which is born with the second component
  and dies when the two components merge).
\end{ex}

Suppose that $X=S^1$ and $Y$ is a closed Riemannian manifold.  Then the (free or based) Lipschitz loop space $\Lip(X,Y)$ is weakly homotopy equivalent to the subset of loops that are arclength-parametrized.  Moreover, on this subspace the Lipschitz constant is just a measure of length, and so is the energy, which famously behaves rather like a Morse function \cite{MilMorse}.  Therefore,  the above formalism is sufficient to analyze the persistent homology of loopspaces.

In general, though, the Lipschitz constant does not behave like a Morse function.  There may not be a discrete set of critical points, and $\Lip_L(X,Y)$ need not have finite-dimensional homology groups.  However, when $X$ and $Y$ are nice finite complexes, this filtration fits into a more general framework described in \cite{CCS,CSGO}.

\subsection{Distance between persistence modules}
In order to describe this framework, we first need to introduce notions of what it means for persistence modules to be ``similar''.  Let $\mathbf M=(\{M_t\},\{\mu(t,s)\})$ and $\mathbf N=(\{N_t\},\{\nu(t,s)\})$ be two persistence modules.  They are \emph{$\delta$-interleaved} if there are $\ph_t:M_t \to N_{t+\delta}$ and $\psi_t:N_t \to M_{t+\delta}$ such that
\begin{enumerate}[(i)]
\item $\ph_t$ and $\psi_t$ are \emph{$\delta$-shifted homomorphisms}, that is,
  \[\ph_s \circ \mu(t,s)=\nu(t+\delta,s+\delta) \circ \ph_t \quad\text{and}\quad \psi_s \circ \nu(t,s)=\mu(t+\delta,s+\delta) \circ \psi_t.\]
\item They are almost inverses, in the sense that
  \[\psi_{t+\delta} \circ \ph_t=\mu(t,t+2\delta) \quad\text{and}\quad \ph_{t+\delta} \circ \psi_t=\nu(t,t+2\delta).\]
\end{enumerate}
The \emph{interleaving distance} between $\mathbf M$ and $\mathbf N$ is the infimum of the $\delta$ for which they are $\delta$-interleaved.  This is defined for all persistence modules (although it can be infinite if no $\delta$-matching exists for any $\delta$).

This is related to another measure of distance, defined between two barcodes, which (we recall) are multisets of pairs $(t_i,s_i)$ of elements of $[-\infty,\infty]$ with $t_i \leq s_i$.  Let us say that a \emph{$\delta$-matching} between two barcodes $D_1$ and $D_2$ is a partial matching between their elements such that:
\begin{enumerate}[(i)]
\item Every pair $(t,s) \in D_1$ with $s-t \geq 2\delta$ is matched to a pair in
  $D_2$, and vice versa.
\item Every matching is between pairs of $L^\infty$-distance at most $\delta$.
\end{enumerate}
Intuitively, every bar (interval) is matched to another bar whose endpoints are at most $\delta$ away, and bars of length $<2\delta$ can be matched to an empty interval and disappear.

The \emph{bottleneck distance} between $D_1$ and $D_2$ is the least $\delta$ such that there is a $\delta$-matching between them.  It is easy to see that a $\delta$-matching between the barcodes associated to two finite-rank persistence modules defines a $\delta$-interleaving between them.  It is a fundamental result that the reverse is also true:
\begin{thm}[{Isometry theorem \cite{CCSGGO}}]
  The interleaving distance between two finite-rank persistence modules is
  equal to the bottleneck distance between their barcodes.
\end{thm}
Further down, we will see how this is extended beyond the finite-rank case.

\subsection{Theory of infinite rank persistence modules}
We will now relax the finiteness assumptions for persistence modules, following \cite{CCS,CSGO}.  In the language of these papers, a persistence module $\mathbf M=(\{M_t\},\{\mu(t,s)\})$ is \emph{q-tame} if for every $t<s$, $\mu(t,s)(M_t) \subset M_s$ is finitely generated.  We will also say that $\mathbf M$ is \emph{$\epsi$-q-tame} if $\mu(t,s)(M_t)$ is finitely generated whenever $s-t>\epsi$.  Evidently:
\begin{prop}
  If $\mathbf M$ is q-tame and $\mathbf N$ is $\epsi$-interleaved with
  $\mathbf M$, then $\mathbf N$ is $2\epsi$-q-tame.
\end{prop}

A q-tame persistence module need not be decomposable into cyclic modules.  Nevertheless, they are well-understood from the point of view of discarding ``ephemeral'' structure (that which persists for zero time).  Formally, $\mathbf M=(\{M_t\},\{\mu(t,s)\})$ is an \emph{ephemeral} persistence module if $\mu_{t,s}=0$ for all $t<s$.  Then we have:
\begin{thm}[{\cite[Thm.~4.5 and Cor.~3.5]{CCS}}]
  The following are equivalent for a pair of q-tame persistence modules:
  \begin{enumerate}[(i)]
  \item They are isomorphic in the ``observable category'', that is, the
    category of persistence modules formally inverting morphisms with ephemeral
    kernel and cokernel.
  \item Their interleaving distance is zero.
  \end{enumerate}
  Moreover, every equivalence class of q-tame persistence modules under this
  relation contains a minimal element under inclusion which is a direct sum of
  interval modules.
\end{thm}
Thus q-tame persistence modules have a well-defined barcode.  This barcode is given an alternate definition in \cite{CSGO} which directly uses the data of the original module.  That paper also generalizes the isometry theorem above:
\begin{thm}[{\cite[Theorem 5.14]{CSGO}}]
  The interleaving distance between two q-tame persistence modules is equal to
  the bottleneck distance between their barcodes.
\end{thm}
In order to have some way of understanding $\epsi$-q-tame modules, we also consider the \emph{$\epsi$-smoothing} of a persistence module $\mathbf M=(\{M_t\},\{\mu(t,s)\})$, defined in \cite[\S5.5]{CSGO} as the module $\mathbf M^\epsi=(\{M^\epsi_t\},\{\mu^\epsi(t,s)\})$ with
\begin{align*}
  M^\epsi_t &= \mu(t-\epsi,t+\epsi)(M_{t-\epsi}) \\
  \mu^\epsi(t,s) &= \mu(t+\epsi,s+\epsi)|_{M^\epsi_t}.
\end{align*}
Evidently, $\mathbf M^\epsi$ is $\epsi$-interleaved with $\mathbf M$.  Moreover, if $\mathbf M$ is $q$-tame, or $\epsi'$-q-tame for any $\epsi'<2\epsi$, then $\mathbf M^\epsi$ has pointwise finite rank (i.e.~each $M^\epsi_t$ has finite rank) and hence decomposes into interval modules by \cite{C-B}.  This means that an $\epsi'$-q-tame module $\mathbf M$ has a well-defined ``barcode of bars of length $>2\epsi$'', in bijection with the bars of $\mathbf M^\epsi$, despite being potentially very infinite at small scales.  This can again be given a precise definition internal to $\mathbf M$ using the machinery of \cite{CSGO}.

\subsection{Persistent homology of Lipschitz functions}
Here we show that when $X$ and $Y$ are nice spaces, then the persistence module $PH_*(\Lip(X,Y),\log\Lip;R)$ fits well into the theory described above.  We omit coefficients, which may be any field.  We first describe situations in which the persistent homology is q-tame:
\begin{thm} \label{tameness}
  Let $X$ be a compact Riemannian manifold with boundary or a finite simplicial
  complex with a simplexwise linear metric, and let $Y$ be a finite simplicial
  complex with a locally CAT($K$) metric for some $K>0$.  Then for each
  $L,\epsi>0$ and every coefficient ring $R$,
  $i_*(L,L+\epsi)(H_*(\Lip_L(X,Y);R)$ is finitely generated, i.e.,
  $PH_*(\Lip(X,Y),\Lip;R)$ is q-tame.
\end{thm}
The property of being q-tame is invariant under reparametrization, so our height function here can be either $\Lip$ or $\log\Lip$.

Note that the CAT($K$) condition is satisfied for some $K>0$ by every closed Riemannian manifold and, more generally, every compact Riemannian manifold with boundary with the induced length metric.  (In the latter case, metric geodesics are concatenations of Riemannian geodesic segments in the interior and in the boundary.)  Moreover, every simplicial complex admits a locally CAT(1) metric \cite[\S12C]{AKP}.

It is \emph{not} satisfied, for example, for surfaces with cone points of total angle $<2\pi$.
\begin{proof}
  If $X$ is a compact Riemannian manifold, it is $(1+\epsi/2L)$-bilipschitz to
  a finite simplicial complex with a simplexwise linear metric (via a fine
  triangulation), so it suffices to assume $X$ is of the latter type.  We will
  show that, inside $\Lip_{L+\epsi}(X,Y)$, $\Lip_L(X,Y)$ can be deformed into
  the image of a finite simplicial complex.

  This finite simplicial complex is constructed as follows.  Let $n=\dim X$,
  and let $r$ be a small radius depending on $K$ and $\epsi$, to be determined
  later.  Choose a subdivision $\mathcal T$ of $X$, depending on $r$ and $L$,
  such that for each $k$-simplex $\tau$ of $\mathcal T$, the linear map
  $f:\Delta^k \to \tau$ satisfies
  \[c(X)\frac{r}{L}d(p,q) \leq d(f(p),f(q)) \leq \frac{r}{L}d(p,q)\]
  for a constant $c(X)>0$.\footnote{There are several ways of choosing such
    subdivisions, see \cite[\S2]{CDMW}.}  Fix an ordering
  $\mathcal T^{(0)}=\{u_1,\ldots,u_M\}$ of the vertices of $\mathcal T$.  Also
  fix a triangulation $\mathcal T'$ of $Y$ such that the diameter of the
  simplices is at most $d_Y=\frac{\epsi c(X)r}{8nL}$.  Then restrictions of
  functions $X \to Y$ to $\mathcal T^{(0)}$ form the complex $(\mathcal T')^M$,
  and the $L$-Lipschitz functions land in the subcomplex
  $\mathbb D \subseteq (\mathcal T')^M$ of tuples $(y_1,\ldots,y_M)$ with the
  property that if $\{u_i,u_j\}$ is an edge of $\mathcal T$, then the simplices
  of $\mathcal T'$ containing $y_i$ and $y_j$ are at distance at most
  $Ld(u_i,u_j)$.

  Now we construct an embedding $F:\mathbb D \to \Lip(X,Y)$ such that
  \[F(y_1,\ldots,y_M)(u_i)=y_i.\]
  Denote the simplex of $\mathcal T$ spanned by $u_{i_0},\ldots,u_{i_k}$ by
  $\tau(u_{i_0},\ldots,u_{i_k})$, and assume that the indices are in ascending
  order.  By induction on $k$, for each $p \in \tau(u_{i_0},\ldots,u_{i_{k-1}})$
  we define
  \[F(y_1,\ldots,y_M)((1-t)p+tu_{i_k})=\gamma_{p,i_k}(t),\]
  where $\gamma_{p,i_k}:[0,1] \to Y$ is the geodesic from $F(y_1,\ldots,y_M)(p)$
  to $y_{i_k}$. For small enough $r$, this geodesic is unique since $Y$ is
  locally CAT($K$).

  \begin{lem}
    We can choose $r$ so that the function $F$ lands in $\Lip_{L+\epsi/2}(X,Y)$.
  \end{lem}
  \begin{proof}
    We would like to show that $f=F(y_1,\ldots,y_M)$ is an
    $(L+\epsi)$-Lipschitz function.

    We start by estimating the length of $\gamma_{p,i_k}$ by induction on $k$;
    we would like to make sure it is at most $(L+\epsi)d(p,u_{i_k})$.  If $k=1$,
    then $\gamma_{p,i_1}$ is the geodesic between two vertices of $\mathcal T'$,
    with length bounded by
    \[Ld(u_{i_0},u_{i_1})+2d_Y \leq \left(L+\frac{\epsi}{4n}\right)d(u_{i_0},u_{i_1}).\]
    At the $k$th step, $\gamma_{p,i_k}$ passes through a vertex and opposite
    edge of a triangle whose sides are geodesics constructed in the $(k-1)$st
    step.  If these sides are short enough, depending on $K$, then
    $\gamma_{p,i_k}$ is at most $1+\epsi/4n$ times the length of the
    corresponding Euclidean geodesic.  By induction, we have
    \[\operatorname{length}(\gamma) \leq \left(1+\frac{\epsi}{4nL}\right)^nLd(p,u_{i_k}) = (L+\epsi/4+O(\epsi^2/L))d(p,u_{i_k})\]
    for a small enough initial $r=r(K,\epsi)$.  This bounds the distances
    between points along $\gamma$.

    Now we would like to bound the distances between a pair of points
    $f(q),f(q')$ lying on different $\gamma_{p,i_k},\gamma_{p',i_k}$ where $p,p'$
    lie in the same $(k-1)$-simplex of $\mathcal T$.  Notice that $f(q)$ and
    $f(q')$ lie on the sides of a triangle formed by $\gamma$, $\tilde\gamma$,
    and a geodesic between two $p$ and $p'$.  By a similar induction, we have
    that
    \[d(f(q),f(q')) \leq \left(L + \frac{\epsi}{4}\cdot\frac{n+1}{n} + O\left(\frac{\epsi^2}{L}\right)\right)d(q,q').\]
    Thus $f$ is $(L+\epsi)$-Lipschitz when restricted to every $k$-simplex, and
    hence everywhere.
  \end{proof}

  Finally, given an $L$-Lipschitz map $f:X \to Y$, there is a corresponding map
  $\tilde f \in F(\mathbb D)$ which coincides with $f$ on the vertices of
  $\mathcal T$.  If $r$ is small enough, we can take a linear homotopy $h_t$
  between $f$ and $\tilde f$.  Moreover, this homotopy goes through
  $(L+\epsi)$-Lipschitz maps by a similar argument to the above.  Let $p,q$ be
  two points in the same simplex of $\mathcal T$.  Their images under the
  homotopy lie on geodesic segments between $f(p)$ and $\tilde f(p)$ and
  between $f(q)$ and $\tilde f(q)$.  These segments have length at most
  $r(1+\epsi/2L)$, and their endpoints are at most $Ld(p,q)$ and at most
  $(L+\epsi/2)d(p,q)$ apart, respectively.  The CAT($K$) condition bounds the
  degree to which distance between geodesics can be nonconvex, so if $r$ is
  small enough,
  \[d(h_t(p),h_t(q)) \leq (L+\epsi)d(p,q). \qedhere\]
\end{proof}

Now we show that things don't get much worse when we replace the spaces in Theorem \ref{tameness} by Lipschitz homotopy equivalent ones:
\begin{thm} \label{thm:interleaving}
  Consider the space $\Lip(X,Y)$ of Lipschitz functions between two metric
  spaces $X$ and $Y$.
  \begin{enumerate}[(a)]
  \item If $X'$ is Lipschitz homotopy equivalent to $X$, then the interleaving
    distance between $PH_*(\Lip(X,Y),\log\Lip)$ and $PH_*(\Lip(X',Y),\log\Lip)$
    is finite.
  \item If $Y'$ is Lipschitz homotopy equivalent to $Y$, then the interleaving
    distance between $PH_*(\Lip(X,Y),\log\Lip)$ and $PH_*(\Lip(X,Y'),\log\Lip)$
    is finite.
  \end{enumerate}
\end{thm}
As an immediate corollary, we get some invariance properties of the persistent homology with respect to Lipschitz homotopy equivalence:
\begin{cor}
  \ 
  \begin{enumerate}[(a)]
  \item If $X$ and $Y$ are metric spaces Lipschitz homotopy equivalent to
    finite complexes, then $PH_*(\Lip(X,Y),\log\Lip)$ is $\epsi$-q-tame for
    some $\epsi<\infty$.
  \item The property that all bars, or all finite bars, in
    $PH_n(\Lip(X,Y),\log\Lip)$ are of bounded length is (Lipschitz) homotopy
    invariant.
  \end{enumerate}
\end{cor}
Moreover, for manifolds we get:
\begin{cor}
  Let $X$, $Y$, $X'$, and $Y'$ be compact Riemannian manifolds with boundary
  such that $X \simeq X'$ and $Y \simeq Y'$.  Then the interleaving distance
  between $PH_*(\Lip(X,Y),\log\Lip)$ and $PH_*(\Lip(X',Y'),\log\Lip)$ is finite.
\end{cor}
\begin{proof}
  It is easy to see that any homotopy equivalence between compact Riemannian
  manifolds can be deformed to a Lipschitz homotopy equivalence.  This reduces
  the corollary to Theorem \ref{thm:interleaving}.
\end{proof}
\begin{proof}[Proof of Theorem \ref{thm:interleaving}]
  We prove this for a Lipschitz homotopy equivalence between $X$ and $X'$; the
  proof for a Lipschitz homotopy equivalence of the target space is identical.
  Let $\xymatrix{X \ar@/^/[r]^f & X' \ar@/^/[l]^g}$ be a Lipschitz homotopy
  equivalence, with $g \circ f \simeq \id$ via a Lipschitz homotopy
  $H:X \times [0,1] \to X$, and $f \circ g \simeq \id$ via a Lipschitz homotopy
  $H':X' \times [0,1] \to X'$.

  Then $f$ induces a map $\Lip(X',Y) \to \Lip(X,Y)$ which sends the sublevel
  set $\Lip_{\leq L}(X',Y)$ to $\Lip_{\leq \Lip(f)L}(X,Y)$.  Similarly, $g$
  induces a map $\Lip(X,Y) \to \Lip(X',Y)$ which sends $\Lip_{\leq L}(X',Y)$ to
  $\Lip_{\leq \Lip(g)L}(X,Y)$.  The composition of these induces a map
  \[\ph_{X'}:\Lip_{\leq L}(X',Y) \to \Lip_{\leq \Lip(g)\Lip(f)L}(X',Y).\]
  This map $\ph_{X'}$ is homotopic to the identity inclusion via the homotopy
  $\ph_t(u)=u \circ H'_t$ inside the perhaps larger sublevel set
  $\Lip_{\leq L_{H'} L}(X',Y)$, where $L_{H'}$ is the maximal Lipschitz constant
  of a fiber $H'_t:X' \to X'$ of $H'$.  Therefore $\ph_{X'}$ and the identity
  inclusion induce the same homomorphism
  \[H_*(\Lip_{\leq L}(X',Y)) \to H_*(\Lip_{\leq L_{H'}L}(X',Y)).\]

  A similar computation holds going in the other direction, where we can
  similarly define
  \[L_H=\max \{\Lip(H_t): t \in [0,1]\}.\]
  Thus the persistence modules of the two filtrations are
  $\log\left(\max\{L_H',L_H\}\right)$-interleaved.
\end{proof}

\section{Maps to manifolds of nonpositive curvature} \label{S3}

In this section we discuss the case when the target space is a nonpositively curved manifold.  In this case, there are no finite bars in the persistent homology of the mapping space:
\begin{thm}
  Let $Y$ be a complete manifold of nonpositive curvature (or, more generally,
  a complete locally CAT(0) metric space), and $X$ any Riemannian manifold (or,
  more generally, a metric space admitting a universal cover).  Then every
  cycle in $\Lip_{\leq L}(X,Y)$ which is trivial in $\Lip(X,Y)$ is trivial in
  $\Lip_{\leq L}(X,Y)$.
\end{thm}
Note that this property is not an invariant of the topology of $Y$.  Take for example $X=S^1$ and $Y$ a nonpositively curved surface.  Adding a small bubble to the metric on $Y$ immediately generates an infinite number of finite bars, one for each multiple of the resulting nullhomotopic closed geodesic.  On the other hand, the proof of Theorem \ref{thm:interleaving} shows that the property ``no finite bars'' is closed in the topology on $\operatorname{Met}(X) \times \operatorname{Met}(Y)$ induced by the Lipschitz distance.
\begin{proof}
  We show the following stronger statement: the intersection of every connected
  component of $\Lip(X,Y)$ with $\Lip_{\leq L}(X,Y)$ is either empty or its
  inclusion into the corresponding component of $\Lip(X,Y)$ is a homotopy
  equivalence.

  Since $Y$ is nonpositively curved, its universal cover is contractible,
  i.e.\ $Y$ is a $K(\Gamma,1)$ for some group $\Gamma$.  A connected component
  of the space of (unbased) maps $X \to Y$ corresponds to a conjugacy class of
  homomomorphisms $\ph:\pi_1(X) \to \Gamma$; call this component $K_{[\ph]}$.

  Let $\tilde X$ and $\tilde Y$ be the universal covers of $X$ and $Y$, and let
  $\widetilde{K_{[\ph]}} \subset \Lip(\tilde X,\tilde Y)$ be the subspace
  consisting of lifts of maps in $K_{[\ph]}$.  Two lifts of the same $f$,
  differing by $g \in \Gamma$, are homotopic if and only if $g\ph g^{-1}=\ph$,
  that is, if $g$ lies in the centralizer $Z(\ph(\pi_1X)) \subseteq \Gamma$.
  (Conversely, if $\tilde f$ is a lift of $f$, then its translate
  $g\cdot\tilde f$ is a lift of $f$ if and only if $g \in Z(\ph(\pi_1X))$; in
  general, $g\cdot\tilde f$ is a lift of a conjugate of $f$.)

  Recall that $\tilde Y$ is uniquely geodesic, and moreover, if
  $\gamma,\gamma':[0,1] \to \tilde Y$ are geodesics, then
  \[d(\gamma(t),\gamma'(t)) \leq \max\{d(\gamma(0),\gamma'(0)),d(\gamma(1),\gamma'(1))\}, \qquad t \in [0,1].\]
  In particular, for every $g \in \Gamma$ and lift $\tilde f$ of a map
  $f \in K_\ph$, there is a ``linear'' equivariant homotopy between $\tilde f$
  and $g\cdot\tilde f$, and this homotopy goes through maps of Lipschitz
  constant $\leq \Lip f$.  This shows:
  \begin{enumerate}
  \item Each connected component of $\widetilde{K_{[\ph]}}$ is a convex, hence
    contractible set, and $K_{[\ph]}$ is a quotient of this set by the action of
    $Z(\ph(\pi_1X))$.
  \item The same is true for the subset of $\widetilde{K_{[\ph]}}$ consisting of
    $L$-Lipschitz maps, so long as one such map exists.
  \end{enumerate}
  It follows that $K_{[\ph]}$ and the set of $L$-Lipschitz maps inside it are
  both $K(Z(\ph(\pi_1X)),1)$ spaces and the inclusion of one into the other is
  a homotopy equivalence.
\end{proof}

\section{Loop spaces of closed manifolds following Nabutovsky--Rotman} \label{S4}

Nabutovsky and Rotman \cite{NabRot} conducted a detailed analysis of what we now recognize as the persistent \emph{homotopy} groups of loop spaces of simply connected closed manifolds.  In this section, we use their results to prove a similar theorem for the persistent homology groups.  Later, we will use a similar outline to prove Theorem \ref{main:pw}; we include this proof in part as a warmup for that argument.

The result we prove is as follows:
\begin{thm} \label{thm:ph-NR}
  Let $M$ be a simply connected closed $n$-manifold.  Let $d$ be its diameter,
  and let $S$ be such that any loop of length $\ell \leq 2d$ can be contracted
  through curves of length at most $\ell+S$.  Then for every $m \geq 1$:
  \begin{enumerate}[(i)]
  \item Every infinite bar in $PH_m(\Omega M,\operatorname{len})$ has birth
    time at most $(6d+2S){m \choose 2}+5d+S$. \label{ph-NR:inf}
  \item Every finite bar in $PH_{m-1}(\Omega M,\operatorname{len})$ has length
    at most $(6d+2S){m \choose 2}+5d+S$.
  \end{enumerate}
\end{thm}
\begin{rmks}
  \begin{enumerate}[(a)]
  \item The lengths of bars are with respect to the length, not the log-length
    functional.  In other words, increasing the length of loops by an additive
    constant is sufficient to kill trivial cycles.
  \item Since the space of loops of length $\leq 2d$ is compact, an $S$ as in
    the hypotheses exists for every $M$.
  \item Theorem \ref{thm:ph-NR} is almost certainly not optimal: it should be
    possible to prove a result with constants linear in $m$ using the
    \emph{methods} of Nabutovsky and Rotman and not just their results.
    Indeed, a version of part \ref{ph-NR:inf} with a linear (albeit inexplicit)
    constant is due to Gromov \cite[\S4]{GrHED}, as is an analogous theorem for
    free loop spaces \cite[Theorem 7.3]{GrMS}; in the next section, we extend
    his method to the relative case.
  \end{enumerate}
\end{rmks}
\begin{proof}
  The proof relies on the following result:
  \begin{thm}[{Nabutovsky--Rotman \cite[Thm.~8.2]{NabRot}}] \label{thm:NR}
    Let $M$ be a simply connected closed $n$-manifold.  Let $d$ be its
    diameter, and let $S$ be as above.  Denote the space
    of loops of length $\leq L$ by $\Omega^LM$.  Then for every $\epsi>0$:
    \begin{enumerate}[(i)]
    \item Every map $S^m \to \Omega M$, $m \geq 1$, can be homotoped into
      $\Omega^{(6d+2S)m-d-S+\epsi}M$.
    \item For $m \geq 1$, every nullhomotopy
      $f:(D^{m+1},\partial D^{m+1}) \to (\Omega M, \Omega^LM)$ of a map
      $S^m \to \Omega^L M$ can be homotoped relative to $\partial D^{m+1}$ into
      $\Omega^{\max\{L,5d+S\}+(6d+2S)m+\epsi}(M)$.  In the case $m=0$, the constant
      is $L+2d+S+\epsi$.
    \end{enumerate}
  \end{thm}
  Now let $(Z,\partial Z)$ be an $m$-pseudomanifold with boundary, and consider
  a simplicial chain $(c,\partial c):(Z,\partial Z) \to (\Omega M,\Omega^L M)$.
  (Here and later we work with integral chains, but the proofs make sense with
  any coefficient ring.) We apply Theorem \ref{thm:NR} inductively to the
  simplices of this chain.  Over the course of the induction, we create a
  sequence of homotopic maps $c_k$.

  To build $c_0$, we homotope $c$ so that each vertex of $Z$ outside
  $\partial Z$ maps to the constant loop.  This homotopy can be extended to all
  of $Z$ via the homotopy extension property.

  Now we perform the inductive step, building a homotopy of $c_k$ to $c_{k+1}$.
  By induction, $c_k$ maps the $k$-skeleton of $Z$ to the space of loops of
  length at most $L+(6d+2S){k \choose 2}+5d+S+\epsi$.  We build $c_{k+1}$ so
  that it is identical to $c_k$ on $Z^{(k)} \cup \partial Z$; on $(k+1)$-cells,
  we apply Theorem \ref{thm:NR} to ensure that $c_{k+1}$ maps to
  $\Omega^{L+(6d+2S){k+1 \choose 2}+5d+S+\epsi}M$; and on higher cells, we
  extend by the homotopy extension property.  At the end of the induction, we
  have built $c_m$, homotopic to $c$ rel $\partial Z$, which maps $Z$ into the
  space of loops of length at most $L+(6d+2S){m \choose 2}+5d+S+\epsi$.

  If $\partial Z$ is empty, we can start the induction with $L=0$.  This gives
  a map into the space of loops of length at most
  $(6d+2S){m \choose 2}+5d+S+\epsi$.
\end{proof}

\section{Loop spaces following Gromov} \label{S:loops-Gro}
Now we use a different method---still heavily inspired by Nabutovsky--Rotman \cite{NabRot}, but also by Gromov \cite[Theorem~7.3]{GrMS}---to show the following:
\begin{thm} \label{thm:ph-Gro}
  Let $Y$ be a simply connected finite simplicial complex with a metric
  bilipschitz to a linear one.  Then there is a constant $C=C(Y)$ such that for
  every $m \geq 1$:
  \begin{enumerate}[(i)]
  \item  \label{ph-Gro:inf}
    Every infinite bar in $PH_m(\Omega Y,\operatorname{len})$ or
    $PH_m(\Lambda Y,\operatorname{len})$ has birth time at most $Cm+C$.
  \item Every finite bar in $PH_{m-1}(\Omega Y,\operatorname{len})$ or
    $PH_{m-1}(\Lambda Y,\operatorname{len})$ has length at most $Cm+C$.
  \end{enumerate}
\end{thm}
Note that \ref{ph-Gro:inf} is one half of \cite[Theorem 1.4]{GrHED} and \cite[Theorem~7.3]{GrMS} (the other half is a corresponding linear lower bound).  However, our proof will recover it for free.
\begin{proof}
  Let $(Z,\partial Z)$ be an $m$-pseudomanifold with boundary, and consider a
  simplicial chain $(c,\partial c):(Z,\partial Z) \to (\Omega Y,\Omega^L Y)$.
  We deform $c$ relative to $\partial Z$ to a $\tilde c$ which maps to
  $\Omega^{L+Cm+C}Y$.

  We do this in two steps.  Consider the map
  $f_0:(Z \times S^1, \partial Z \times S^1) \to Y$ induced by $c$; we can
  assume $f_0$ is Lipschitz after an arbitrarily small deformation.  The first
  step is to deform $f_0$ (relative to $\partial Z \times S^1$) to a map $f$
  which behaves in a regular way on a fine triangulation of $Z \times S^1$.  In
  particular, it will send most of the $1$-skeleton of this triangulation to
  the basepoint.  The second step is to reparametrize: choose a map
  $\Gamma:Z \times S^1 \to Z \times S^1$ (homotopic to the identity relative to
  $\partial Z \times S^1$) which maps every $S^1$-fiber mostly to the
  $1$-skeleton.  Since $f$ sends most of the $1$-skeleton to the basepoint, the
  composition $f \circ \Gamma$ will have short $S^1$-fibers.  Then we can set
  \[\tilde c(z)(t)=f \circ \Gamma(z,t).\]

  We start with the first step.  Given triangulations of $Z$ and $S^1$, the
  induced cell structure on the product can be subdivided into a triangulation
  without adding any vertices. We choose fine enough triangulations so that:
  \begin{enumerate}
  \item $\partial Z$ is the full subcomplex of $Z$ spanned by its vertices.
  \item The map sending each vertex $v$ of $Z \times S^1$ to the vertex of $Y$
    nearest to $f_0(v)$ extends to a simplicial approximation of $f_0$.
  \item Fixing the cellwise linear metric $d_\Delta$ on $Z \times S^1$ in which
    all edges have length $1$, $f_0$ is $1$-Lipschitz on
    $\partial Z \times S^1$.
  \end{enumerate}
  Now let $f_t$ be a homotopy from $f_0$ to a map $f_1$ as follows:
  \begin{description}
  \item[Times $0 \leq t \leq \frac{1}{2}$] Apply the linear homotopy between
    $f_0$ and its simplicial approximation.
  \item[Times $\frac{1}{2} \leq t \leq 1$] Compose with a homotopy
    $Y \times [\frac{1}{2},1] \to Y$ from the identity to a map sending
    $Y^{(1)}$ to the base point.
  \end{description}
  Then $f_t:Z \times S^1 \times [0,1] \to Y$ is $C(Y)$-Lipschitz with
  respect to the product metric $d_\Delta \times d_{[0,1]}$.
  
  Finally, let $f(z,s)=f_{t(z)}(z,s)$, where
  \begin{itemize}
  \item $t(z)=1$ if $z$ is contained in a simplex disjoint from $\partial Z$;
  \item If the simplex containing $z$ is $\Delta * \Delta'$, where
    $\Delta \subset \partial Z$ and $\Delta'$ is disjoint from $\partial Z$,
    then $t(z)$ is such that $z=(1-t(z))v+t(z)v'$, for $v \in \Delta$ and
    $v' \in \Delta'$.
  \end{itemize}
  Then $f$ is again $C(Y)$-Lipschitz with respect to $d_\Delta$.

  Now $f$ takes all internal edges of $Z \times S^1$ to the basepoint.  So
  given a path $\gamma:[0,1] \to Z \times S^1$, the length of $f \circ \gamma$
  is determined by the portion of $\gamma$ that does not traverse internal
  edges (i.e.~those disjoint from $\partial Z$).  The next step is to build a
  map $\Gamma:Z \times S^1 \to Z \times S^1$ which is the identity on
  $\partial Z \times S^1$ and such that $f \circ \Gamma$ has short $S^1$-fibers.

  To build intuition, assume first that $Z$ is $1$-dimensional, and consider an
  edge $e(v,w)$ between vertices $v,w \in Z$.  Let $u_0,u_1,\ldots,u_N=u_0$ be
  the vertices of $S^1$ under our triangulation.  Let $\gamma_0$ be the
  $S^1$-fiber over $v$, $\gamma_N$ be the $S^1$-fiber over $w$, and $\gamma_i$
  for $i=1,\ldots,N-1$ be the piecewise linear loop in $Z \times S^1$ with
  vertices
  \[(v,u_0),(v,u_1),\ldots,(v,u_i),(w,u_i),\ldots,(w,u_N),(v,u_N)=(v,u_0).\]
  If $v$ and $w$ are both internal vertices, then $f \circ \gamma_i$ has
  length $0$ for every $i$; if one is an external vertex, then it has length
  at most $L+2C(Y)$.  Moreover, $\gamma_i$ and $\gamma_{i+1}$ are homotopic via
  a homotopy that just traverses one grid square; these homotopies assemble
  into a map $\Gamma:[0,N]^2 \to e(v,w) \times [0,N]$, and each fiber of the
  homotopy $f \circ \Gamma$ has length at most $L+3C(Y)$.

  Now we generalize this construction to higher dimensions in a more abstract
  way, inspired by Gromov's proof of \cite[Theorem~7.3]{GrMS}.  We use an
  observation of Milnor \cite{MilUniv}: the set of piecewise linear paths in a
  simplicial complex admits a cell structure in which every path
  $v_1,\ldots,v_k$ with $v_i$ contained in a simplex $\Delta_i$ is a point of
  the cell $\Delta_1 \times \cdots \times \Delta_k$.  We consider this
  construction for $Z \times S^1$.  For each simplex $\Delta$ of $Z$, consider
  the subcomplex of piecewise linear paths
  $\mathcal P_\Delta \subset \{\gamma:[0,N] \to \Delta \times S^1\}$ consisting
  of $\gamma$ such that
  \begin{enumerate}[(i)]
  \item $\gamma([i-1,i]) \subset \Delta \times e_i$, where
    $e_i=[u_{i-1},u_i]$;
  \item $\gamma(0)=\gamma(N) \in \Delta \times \{0\}$.
  \end{enumerate}
  This subcomplex is contractible, indeed convex; therefore, we can homotope
  the map
  \begin{gather*}
    \Gamma_0:Z \to \mathcal P(Z \times S^1) \\
    \Gamma_0(z)(t)=(z,t)
  \end{gather*}
  via a linear homotopy $\Gamma_s$ to a \emph{cellular} map $\Gamma_1$ that
  still takes each $\Delta$ into $\mathcal P_\Delta$.  Moreover, for
  $\gamma,\gamma' \in \mathcal P_\Delta$, we have
  \[d_\Delta(\gamma(t),\gamma'(t)) \leq m+1, \qquad \text{for all }t \in [0,N];\]
  therefore $\Gamma_s$ moves points by at most $m+1$.  Now we define a map
  $\Gamma:Z \to \mathcal P(Z \times S^1)$ by:
  \begin{itemize}
  \item On interior simplices of $Z$, $\Gamma=\Gamma_1$.
  \item For a simplex $\Delta*\Delta'$ of $Z$, where $\Delta$ is in
    $\partial Z$ and $\Delta'$ is disjoint from $\partial Z$, write points as
    $z=(1-s)v+sv'$, where $v \in \Delta$ and $v' \in \Delta'$.  Then,
    identifying $S^1$ with $[0,N]/0 \sim N$, we define
    \[\Gamma(z)(t)=\begin{cases}
        \Gamma_1(v')(t) & |t-N/2| \leq (N/2+1)s-1 \\
        (v,t) & |t-N/2| \geq (N/2+1)s \\
        \Gamma_{s'}((1-s')v+s'v')(t) & s'=(N/2+1)s-|t-N/2| \in (0,1).
      \end{cases}\]
  \end{itemize}
  We claim that $\tilde c(z)(t)=f(\Gamma(z)(t))$ is the desired relative
  cycle.  We need only bound the length of each path $\tilde c(z)$.  But
  notice:
  \begin{itemize}
  \item $\Gamma_1(z)$ lies in the $m$-skeleton of the path space.  This means
    that it lies in a cell $\Delta_1 \times \cdots \times \Delta_k$ where at
    most $m$ of the $\Delta_i$ have dimension $\geq 1$.  Suppose now that $z$
    lies in an interior simplex of $Z$.  Since $f$ maps all interior edges to
    the base point, the curve $f \circ \Gamma_1(z)$ has length at most
    $m\Lip_{d_\Delta}(f)$.
  \item For $z \in \partial Z$, the curve $f(z,t)$ has length at most $L$.
  \item The speed of the remaining segment of $\Gamma(z)$ (when $z$ is not
    contained in an interior or a boundary simplex) is always at most $m+2$,
    and therefore that segment has length at most $2m+4$.
  \end{itemize}
  In conclusion, $\operatorname{len}(\tilde c(z)) \leq L+(3m+4)\Lip_{d_\Delta}(f)$.
\end{proof}

\section{Tools from quantitative homotopy theory} \label{S5}

The rest of the paper focuses on simply connected and, more generally, nilpotent target spaces $Y$, where Sullivan's model of rational homotopy theory applies.  Here we are able to use a number of tools from quantitative homotopy theory developed in the last decade.

\subsection{Rational homotopy theory}
We start with a very brief review of Sullivan's theory of minimal models, focused on fixing notation.  We refer the reader to \cite{GrMo,FHT} for more details on the general background and \cite{PCDF,scal} for treatments geared towards quantitative topology.

There is a \emph{rationalization} functor on homotopy types of simply connected spaces.  There are several ways to define this, but perhaps the easiest is by tensoring all spaces and maps in the Postnikov tower with $\QQ$.  We say two spaces $X$ and $Y$ are \emph{rationally equivalent}, written $X \simeq_{\QQ} Y$, if they have homotopy equivalent rationalizations.

Rational homotopy theory provides a way of translating the topology of simply connected spaces into algebraic language, which preserves the same information as the rationalization functor.  There are several equivalent such languages, but the main one we will use is that of differential graded algebras, as developed by Sullivan.

A \emph{(commutative) differential graded algebra}, or \emph{DGA}, is a cochain complex over a field, typically $\QQ$ or $\RR$, with a graded-commutative multiplication satisfying the graded Leibniz rule.  The prototypical examples are:
\begin{itemize}
\item The smooth forms $\Omega^*(X)$ on a smooth manifold $X$, or the
  simplexwise smooth forms on a simplicial complex.
\item The \emph{flat forms} $\Omega^*_\flat(X)$ on a smooth manifold or
  simplicial complex $X$.  This is the closure of the smooth or piecewise
  smooth forms under the flat norm, defined by Whitney \cite[Ch.~IX]{GIT}.  It
  has many of the same properties as smooth forms, and additionally is
  preserved under pullback by Lipschitz maps.
\item Sullivan's \emph{minimal DGA} $\mathcal{M}_Y^*(\mathbb{F})$ (where
  $\mathbb F=\RR$ or $\QQ$) for a simply connected space $Y$.  This is a free
  graded commutative algebra generated in degree $n$ by a vector space of
  \emph{indecomposable} elements $V_n=\Hom(\pi_n(Y);\mathbb{F})$ and with a
  differential which takes elements of $V_n$ to elements of
  $\Lambda_{k=2}^{n-1} V_k$ and is dual to the $k$-invariants in the Postnikov
  tower of $Y$, $k_n \in H^{n+1}(Y_{n-1};\pi_n(Y))$.  We will write
  \[\mathcal M_Y^*=\mathcal M_Y^*(\RR) \cong \Lambda_{n=2}^\infty V_n,\]
  noting that this isomorphism is non-canonical.  We also write
  \[\mathcal M_Y^*(n)=\Lambda_{k=2}^n V_k;\]
  this is the minimal DGA of the $n$th Postnikov stage of $Y$.
\end{itemize}
A \emph{quasi-isomorphism} between DGAs is a map inducing an isomorphism on cohomology.  The existence of such a map between $\mathcal A$ and $\mathcal B$ is not an equivalence relation; therefore we say that two DGAs are \emph{quasi-isomorphic} if they are connected by a zig-zag of one or more quasi-isomorphisms
\[\mathcal A \leftarrow \mathcal C_1 \rightarrow \cdots \leftarrow \mathcal C_k \rightarrow \mathcal B.\]

A \emph{homotopy} between DGA homomorphisms $\ph,\psi:\mathcal A \to \mathcal B$ is a homomorphism
\[\eta:\mathcal A \to \mathcal B \otimes \Lambda(t,dt),\]
where $\Lambda(t,dt)$ is formally generated by a degree $0$ generator $t$ and its differential, but can also be thought of as the algebra of polynomial differential forms on the unit interval.  For DGAs over $\RR$, we get the same homotopy theory by using all differential forms on the interval instead.

If $Y$ is a simply connected smooth manifold or simplicial complex, then it has a (non-unique) \emph{minimal model}, that is, a quasi-isomorphism $m_Y:\mathcal M_Y^* \to \Omega^*(Y)$ realizing the generators of the minimal DGA as differential forms.  The codomain of the minimal model may just as well be $\Omega^*_\flat(Y)$, see \cite[\S6]{scal}; in fact, we use flat forms implicitly throughout this paper so that we can pull them back along Lipschitz maps, but omit the $\flat$ symbol to save on notation.

We write $\mathcal M_Y^*(k)$ for the subalgebra generated by $V_{\leq k}$.  This is the algebraic equivalent of the $k$th Postnikov stage, and maps from $\mathcal M_Y^*(k)$ admit an obstruction theory formally dual to that to rational Postnikov stages; see \cite[Ch.~10 and 14]{GrMo} for details.  We cite and use various obstruction-theoretic lemmas of this form in the course of our proofs.

\subsection{The shadowing principle}
The correspondence $f \mapsto f^*m_Y$ defines a mapping
\[\Lip(X,Y) \to \Hom(\mathcal M_Y^*,\Omega^*X)\]
sending genuine maps between $X$ and $Y$ to their ``algebraicization''.  This induces a function between the corresponding sets of homotopy classes; we can think of the codomain of this function loosely as ``homotopy classes tensored with the reals''.  However, here we focus on what happens inside a homotopy class.  The main technical result of \cite{PCDF} shows that the image within a given homotopy class, if it is nonempty, is in some sense coarsely dense.  That is, given a homomorphism $\ph:\mathcal{M}_Y^* \to \Omega^*X$ which is homotopic to the image of a genuine map, one can produce another genuine map $X \to Y$ which is \emph{nearby} $\ph$.  Moreover, the Lipschitz constant of this new map depends on geometric properties of $\ph$.

To state this precisely, we first introduce some definitions.  Let $X$ and $Y$ be finite simplicial complexes or compact Riemannian manifolds such that $Y$ is simply connected and has a minimal model $m_Y:\mathcal{M}_Y^* \to \Omega^*Y$.  Fix norms on the finite-dimensional vector spaces $V_k$ of degree $k$ indecomposables of $\mathcal{M}_Y^*$; then for homomorphisms $\ph:\mathcal{M}_Y^* \to \Omega^*(X)$ we define the formal dilatation
\[\Dil(\ph)=\max_{2 \leq k \leq \dim X} \lVert\ph|_{V_k}\rVert_{\mathrm{op}}^{1/k},\]
where we use the $L^\infty$ norm on $\Omega^*(X)$.  Notice that if $f:X \to Y$ is an $L$-Lipschitz map, then $\Dil(f^*m_Y) \leq CL$, where the exact constant depends on the dimension of $X$, the minimal model on $Y$, and the norms.  Thus (although no reverse inequality holds) the dilatation is an algebraic analogue of the Lipschitz constant.

Given a formal homotopy
\[\Phi:\mathcal{M}_Y^* \to \Omega^*(X \times [0,T]),\]
we can define the dilatation $\Dil_T(\Phi)$ in a similar way.  The subscript indicates that we can always rescale $\Phi$ to spread over a smaller or larger interval, changing the dilatation; this is a formal analogue of defining separate Lipschitz constants in the time and space direction, although in the DGA world they are not so easily separable.

Now we can state some results from \cite{PCDF}.  We start with a simplified version of \cite[Thm.~4--1]{PCDF}:
\begin{thm}[{Shadowing principle, non-relative version}] \label{shadow}
  Let $X$ be a simplicial complex equipped with the standard simplexwise linear
  metric, and let $Y$ be a simply connected compact Riemannian manifold or
  simplicial complex.  Let $\ph:\mathcal{M}_Y^* \to \Omega^*(X)$ be a
  homomorphism such that
  \begin{enumerate}[(i)]
  \item $\Dil(\ph) \leq L$;
  \item $\ph$ is formally homotopic to $f^*m_Y$ for some $f:X \to Y$.
  \end{enumerate}
  Then there is a $g:X \to Y$ such that
  \begin{enumerate}[(i)]
  \item $g$ is $C(\dim X,Y)(L+1)$-Lipschitz;
  \item $g$ is homotopic to $f$;
  \item $g^*m_Y$ is homotopic to $\ph$ via a homotopy $\Phi$ satisfying
    $\Dil_{1/L}(\Phi) \leq C(\dim X,Y)(L+1)$.
  \end{enumerate}
\end{thm}
In other words, one can produce a genuine map by a small formal deformation of $\ph$.  Note that in the above result, $X$ does not have to be compact.  In fact, the constants depend only on the bounds on the local geometry of $X$.

The most general version of the \cite[Thm.~4--1]{PCDF} is relative:
\begin{thm}[{Shadowing principle, general version}] \label{shadow}
  Let $X$ be a simplicial complex equipped with the standard simplexwise linear
  metric, $A$ a subcomplex of $X$, and let $Y$ be a simply connected compact
  Riemannian manifold or simplicial complex.  Let
  $\ph:\mathcal{M}_Y^* \to \Omega^*(X)$ be a homomorphism and $f:X \to Y$ be a
  map such that
  \begin{enumerate}[(i)]
  \item $\Dil(\ph) \leq L$;
  \item $f|_A$ is $L$-Lipschitz;
  \item $\ph|_A=f^*m_Y|_A$;
  \item $\ph$ is formally homotopic to $f^*m_Y$ relative to $\Omega^*A$.
  \end{enumerate}
  Then there is a $g:X \to Y$ such that
  \begin{enumerate}[(i)]
  \item $g$ is $C(\dim X,Y)(L+1)$-Lipschitz;
  \item $g$ is homotopic to $f$ relative to $A$;
  \item $g^*m_Y$ is homotopic relative to $\Omega^*A$ to $\ph$ via a homotopy
    $\Phi$ satisfying $\Dil_{1/L}(\Phi) \leq C(\dim X,Y)(L+1)$.
  \end{enumerate}
\end{thm}

In the rest of the paper, we use both this result directly and a number of its consequences.

\subsection{Formality, scalability and positive weights}
Now we outline several distinguished classes of rational homotopy types.  All of these have multiple characterizations, both in terms of algebraic properties of the minimal model and in terms of maps.

The broadest such class is that of \strong{spaces with positive weights}, a term attributed in \cite{BMSS} to Morgan and Sullivan.  A minimal DGA $\mathcal A$ has positive weights if it satisfies the following equivalent characterizations:
\begin{enumerate}[(i)]
\item $\mathcal A$ admits a second grading with respect to which the
  differential has degree zero.
\item There is a family of automorphisms $\rho_t:\mathcal A \to \mathcal A$
  and a basis for the indecomposables such that $\rho_t$ multiplies every basis
  element by $t^k$ for some $k$.
\end{enumerate}
Simply connected spaces whose rational homotopy type has positive weights form a large class: for example, they include homogeneous spaces \cite[Prop.~3.7]{BMSS} and smooth complex algebraic varieties \cite{Morgan}.  In fact, although in some sense ``almost all'' spaces do not have positive weights, it is somewhat difficult to find a non-example; as far as we know, the lowest-dimensional one is a complex given in \cite[\S4]{MT} constructed by attaching a 12-cell to $S^3 \vee \CC P^2$.

A somewhat narrower class is that of \strong{formal spaces}, discussed by Sullivan in \cite[\S12]{SulLong}.  A minimal DGA $\mathcal A$ is formal if it satisfies the following equivalent characterizations:
\begin{enumerate}[(i)]
\item There is a quasi-isomorphism $\mathcal A \to H^*(\mathcal A)$.
\item $\mathcal A$ admits a second grading with respect to which the
  differential has degree zero, and such that classes in $H^k(\mathcal A)$
  admit representative cycles of degree $k$. \label{cond:grading}
\item There is a family of automorphisms $\rho_t:\mathcal A \to \mathcal A$
  which induces the map $x \mapsto t^kx$ on $H^k(\mathcal A)$. \label{cond:aut}
\end{enumerate}

Simply connected spaces whose rational homotopy type is formal include symmetric spaces \cite{SulLong} and K\"ahler manifolds \cite{DGMS}.

Finally, \strong{scalable spaces} were introduced in \cite{scal} as a metric refinement of formal spaces.  A simply connected Riemannian manifold with boundary $X$ is scalable if it satisfies the following equivalent characterizations \cite{scal,BGM}:
\begin{enumerate}[(i)]
\item It is formal, and there is an embedding $H^*(X;\RR) \hookrightarrow \bigoplus_i \Lambda^*\RR^{n_i}$ for some collection of $n_i$.  (If $X$ is a closed $n$-manifold, then the target can be $\Lambda^*\RR^n$.)
\item There is an embedding $H^*(X;\RR) \hookrightarrow \Omega^*X$ sending each
  cohomology class to a representative.
\item For an infinite family of $t \in \ZZ$, there is a family of
  $O(t)$-Lipschitz self-maps $r_t:X \to X$ which induce the map
  $x \mapsto t^kx$ on $H^k(X;\RR)$.
\end{enumerate}
Scalable spaces include, once again, symmetric spaces, including in particular spheres and complex projective spaces.  Products and wedge sums of scalable spaces are also scalable.  Scalable spaces also include the connected sum of two or three (but not four) copies of $\CC P^2$ or $S^2 \times S^2$.

Scalable spaces have a number of interesting geometric properties.  Of these, the most relevant for this paper is a stronger version of the shadowing principle.  Given a formal minimal DGA $\mathcal A=\Lambda_{k=2}^\infty V_k$, denote by $U_i,i=0,1,\ldots$ the subspaces comprising the second grading implied by formality.  Then, given a homomorphism $\ph:\mathcal A \to \Omega^*X$, where $X$ is a Riemannian manifold or simplicial complex, define
\[\Dil^U(\ph)=\max_{\substack{2 \leq k \leq \dim X\\2 \leq i \leq 2\dim X-2}} \lVert \ph|_{V_k \cap U_i} \rVert_{\mathrm{op}}^{1/i}.\]
We refer to this quantity as the \emph{$U$-dilatation} of $\ph$.
\begin{ex}
  It is easy to see that condition \ref{cond:grading} for formality implies
  condition \ref{cond:aut}: given a second grading
  $\mathcal A=\bigoplus_i U_i$, there is an automorphism that multiplies
  elements of $U_i$ by $t^i$.  This automorphism has $U$-dilatation $t$.

  For $S^n \vee S^n$, the second grading $\bigoplus_i U_i$ of the minimal model
  is unique up to isomorphism.  In this case, the rational homotopy groups form
  a free Lie algebra generated by the inclusions of the two spheres under the
  Whitehead product operation.  In particular, the $i$th order Whitehead
  products span $\pi_{i(n-1)+1}(S^n \vee S^n)$, and the dual $V_{i(n-1)+1}$ is
  also $U_{in}$.  One can see by inductively computing that
  $dV_{i(n-1)+1} \subseteq U_{in+1}$, but also by considering the action of
  degree $d$ self-maps on $S^n \vee S^n$ on the rational homotopy groups. 
\end{ex}
\begin{thm}[{Improved shadowing principle for scalable targets \cite[Lemma 8.3]{scal}}]
  Let $X$ be a simplicial complex equipped with the standard simplexwise linear
  metric, $A$ a subcomplex of $X$, and let $Y$ be a scalable compact Riemannian
  manifold or simplicial complex.  Let $\ph:\mathcal M_Y^* \to \Omega^*X$ be a
  homomorphism and $f:X \to Y$ be a map such that
  \begin{enumerate}[(i)]
  \item $\Dil^U(\ph) \leq L$;
  \item $f|_A$ is $L$-Lipschitz;
  \item $\ph|_A=f^*m_Y|_A$;
  \item $\ph$ is formally homotopic to $f^*m_Y$ relative to $\Omega^*A$.
  \end{enumerate}
  Then there is a $g:X \to Y$ such that
  \begin{enumerate}[(i)]
  \item $g$ is $C(\dim X,Y)(L+1)$-Lipschitz;
  \item $g$ is homotopic to $f$ relative to $A$.
  \end{enumerate}
\end{thm}
The main advantage of this result is that inductive algebraic constructions tend to produce homomorphisms with large dilatation, but small $U$-dilatation.  Therefore the improved shadowing principle can be used to produce maps to a scalable space with the ``best possible'' geometry for their homotopy class; see Gromov \cite[\S7B]{GrMS}.  For formal spaces that are not scalable, such a best-case scenario cannot happen, see \cite{scal,BGM}.

\subsection{Homotopies of maps from scalable spaces}
Now we prove a new lemma in the spirit of \cite{PCDF} which will be useful in \S\ref{S:example}.  This lemma will allow us to find lower bounds on the size of all homotopies between two maps---a very large parameter space!---by studying a much more restrictive space of algebraic homotopies, parametrized by a few numerical variables and one-variable Lipschitz functions.  It holds generally for maps from scalable domains:

\begin{lem} \label{lem:formalize}
  Let $Y$ be simply connected and $X$ a scalable space.  Let
  \[\alpha=\alpha(X,Y)=\prod_{n \leq \dim X} \frac{n+1}{n},\]
  where $n$ ranges over degrees in which $\mathcal M_Y$ has generators with
  nontrivial differential.
  \begin{enumerate}[(i)]
  \item Let $f:X \to Y$ be a map.  Then the diagram
    \[\xymatrix{
        \mathcal M_Y \ar@{-->}[r]^-\ph \ar[d]^{m_Y} & (H^*(X),d=0) \ar@{^(->}[d]^{\rho} \\
        \Omega^*Y \ar[r]^-{f^*} & \Omega^*(X)
      }\]
    can be completed up to a homotopy
    \[\Phi:\mathcal M_Y \to \Omega^*X \otimes \Lambda(t,dt)\]
    so that $\Dil(\ph)=O((\Lip f)^\alpha)$ and $\Dil(\Phi)=O((\Lip f)^\alpha)$.
  \item \label{part:rel}
    Let $h:X \times [0,S] \to Y$ be a homotopy, where $X$ is a scalable
    space and $Y$ is simply connected.  Then the diagram
    \[\xymatrix{
        \mathcal M_Y \ar@{-->}[r]^-\eta \ar[d]^{m_Y} & (H^*(X),d=0) \otimes \Omega^*[0,S] \ar@{^(->}[d]^{\rho \otimes \id} \\
        \Omega^*Y \ar[r]^-{h^*} & \Omega^*(X \times [0,S])
      }\]
    can be completed up to a homotopy
    \[\Phi:\mathcal M_Y \to \Omega^*(X \times [0,S]) \otimes \Lambda(t,dt)\]
    so that $\Dil(\eta)=O((\Lip h)^\alpha)$ and $\Dil(\Phi)=O((\Lip h)^\alpha)$.

    Moreover, $\eta|_{s=0}$ and $\eta|_{s=S}$ can be chosen to be any
    homomorphisms $\mathcal M_Y \to H^*(X)$ which are homotopic to $h_0$ and
    $h_1$, respectively, via a homotopy of dilatation $O((\Lip h)^\alpha)$.
  \end{enumerate}
\end{lem}
We will apply part \ref{part:rel} for $Y=S^3 \vee S^3$ and $7$-dimensional $X$; in this case, the statement yields $\alpha=48/35$.  However, by employing a more fine-grained understanding of the minimal model of $Y$, we can improve this to $\alpha=9/7$, as explained below the proof.  This is the estimate we will actually use.

In order to prove the lemma, we recall \cite[Prop.~3.9]{PCDF}:
\begin{prop} \label{prop:extHtpy}
  Suppose that $\Phi_k:\mathcal{M}_Y^*(k) \to \Omega^*X \otimes \Lambda(t,dt)$
  is a partially defined algebraic homotopy between
  $\ph,\psi:\mathcal{M}_Y^* \to \Omega^*X$.
  \begin{enumerate}[(i)]
  \item \label{num:obst} The obstruction to extending $\Phi_k$ to a homotopy
    \[\Phi_{k+1}:\mathcal{M}_Y^*(k+1) \to \Omega^*X \otimes \Lambda(t,dt)\]
    is a class in $H^{k+1}(X;V_{k+1})$ represented by a cochain
    \[\sigma(v)=\psi(v)-\ph(v)-{\textstyle \int_0^1 \Phi_k(dv)},\]
    and therefore in general
    \[\lVert \sigma \rVert_{\mathrm{op}} \leq
    C(k,d|_{V_{k+1}})\Dil(\Phi_k)^{k+2}
    +\Dil(\ph)^{k+1}+\Dil(\psi)^{k+1}.\]
  \item \label{num:ext}
    If this obstruction class vanishes, then we can choose a primitive $c(v)$
    for $\sigma(v)$ and fix
    \[\Phi_{k+1}(v)=\ph(v)+d(c(v) \otimes t)+{\textstyle \int_0^t \Phi_k(dv)},\]
    so that in general
    \[\bigl\lVert\Phi_{k+1}|_{V_{k+1}}\bigr\lVert_{\mathrm{op}} \leq
    (C_{\mathrm{IP}}+2)\bigl(C(k,d|_{V_{k+1}})
    \Dil(\Phi_k)^{k+2}+\Dil(\ph)^{k+1}+\Dil(\psi)^{k+1}\bigr),\]
  where $C_{\mathrm{IP}}$ is the isoperimetric constant for $(k+2)$-forms in
  $X$.
  \end{enumerate}
  Moreover, if for some subcomplex $A \subset X$ we have an existing homotopy
  \[\chi:\mathcal{M}_Y^* \to \Omega^*A \otimes \Lambda(t,dt)\]
  between $\ph|_A$ and $\psi|_A$, and there is no obstruction to extending it,
  we can get an extension with similar bounds, using a relative isoperimetric
  constant and with an additional $O(\Dil(\chi))$ term.
\end{prop}
\begin{proof}[Proof of Lemma \ref{lem:formalize}]
  We first prove (i), building $\Phi$ by induction on degree.  Suppose we have
  constructed a homotopy $\Phi_k$ in degrees up to $k$.  In particular,
  since at time $1$ the homotopy factors through $\rho$, we have
  $\Phi_k(dV_{k+1})|_{t=1}=0$ and so applying $\Phi|_{t=1}$ to $V_{k+1}$ should
  give us closed forms.  Applying Prop.~\ref{prop:extHtpy}(i), we get that the
  obstruction to making $\Phi_{k+1}(V_{k+1})|_{t=1}=0$ is bounded by
  $\Dil(\Phi_k)^{k+2}$.  We then define $\Phi_{k+1}(V_{k+1})|_{t=1}$ by setting
  $\ph$ to the value of this obstruction.  Then we apply
  Prop.~\ref{prop:extHtpy}(ii) to define $\Phi_{k+1}(V_{k+1})$ for all
  $t \in [0,1]$.

  For (ii), we first use (i) to construct $\Phi$ at all integer times
  $s \in [0,S]$.  Then we extend it to intervals $[s,s+1]$.  Note that going
  along three sides of the square already defines a homotopy between $\eta|_s$
  and $\eta|_{s+1}$, so there is no obstruction to extending it.  We apply the
  relative version of Prop.~\ref{prop:extHtpy} to complete the extension.

  From the proof it is clear that the choice of $\eta|_{s=0}$ and $\eta|_{s=S}$
  doesn't affect the end result.
\end{proof}
The estimate coming from Prop.~\ref{prop:extHtpy} can be improved if we have different estimates on the operator norms of $\Phi_k$ in different degrees.  For example, in the case of $Y=S^3 \vee S^3$, we have $\lVert\Phi_3|_{V_3}\rVert_{\text{op}}=O((\Lip h)^3)$, since the differential at that stage is $0$, and $\lVert\Phi_5|_{V_5}\rVert_{\text{op}}=O((\Lip h)^6)$.  Now each term of the differential of elements of degree $7$ is the product of a degree $3$ and a degree $5$ generator, so its operator norm is $O((\Lip h)^9)$.  This differential is the source of the $\Dil(\Phi_k)^{k+2}$ term in the estimates in Prop.~\ref{prop:extHtpy}, and the other terms are smaller; thus this is also the right estimate for the primitive $\Phi_7|_{V_7}$.  Therefore we get $\alpha=9/7$ in this case.

\section{Bounded length bars} \label{S6}

In this section, we prove that for certain mapping spaces, the length of finite bars in persistent homology with respect to the log-Lipschitz constant is always bounded.  In each case, we show this by proving a stronger result: that every chain in the mapping space whose boundary lies in $L$-Lipschitz functions can be deformed relative to this boundary into the subspace of $CL$-Lipschitz functions.

\subsection{Maps from $k$-dimensional spaces to $k$-connected spaces}

\begin{thm} \label{thm:ph-loops}
  Let $Y$ be a compact, simply connected, rationally $k$-connected Riemannian
  manifold with boundary, and let $X$ be a finite $k$-dimensional simplicial
  complex equipped with a standard simplexwise metric.  Let $A \subset X$ be a
  subcomplex, and fix a Lipschitz map $f:A \to Y$.  Write $\Lip(X,Y)_A$ to mean
  the space of Lipschitz maps $X \to Y$ extending $f$.  Then there are
  constants $C(Y,n)$ such that:
  \begin{enumerate}[(i)]
  \item Every infinite bar in $PH_n(\Lip(X,Y)_A,\log_+\Lip)$ has birth time in
    \[[\log_+\Lip f,\log_+\Lip f+C(Y,n+k)].\]
  \item Every finite bar in $PH_n(\Lip(X,Y)_A,\log_+\Lip)$ has length bounded by
    $C(Y,n+k)$.
  \end{enumerate}
\end{thm}
As a special case we have a weak version of Theorem \ref{thm:ph-Gro}:
\begin{cor}
  Let $Y$ be a compact, simply connected Riemannian manifold with boundary.
  Then the lengths of finite bars in $PH_n(\Omega Y,\log\operatorname{len})$
  and $PH_n(\Lambda Y,\log\operatorname{len})$ (the based and free
  loopspaces, respectively) are bounded by a constant $C(Y,n)$.
\end{cor}
\begin{proof}
  Let $(Z,\partial Z)$ be an $n$-pseudomanifold with boundary, and consider a
  simplicial chain
  \[(c,\partial c):(Z, \partial Z) \to (\Lip(X,Y)_A,\Lip_L(X,Y)_f),\]
  where $\Lip_L$ denotes the subspace of $L$-Lipschitz maps.  We may assume,
  via a perturbation, that this map is Lipschitz with respect to the
  $C^0$-distance; equivalently, it defines a Lipschitz map
  $g_c:Z \times X \to Y$.

  Let $m_Y:\mathcal M_Y \to \Omega^*Y$ be a minimal model for $Y$.  Notice that
  the image of $g_c^*m_Y$ lies in forms of degree $\geq k+1$.  For $i \geq k+1$,
  every $i$-vector in $X$ is trivial.  Therefore, if we rescale the metric on
  $Z$ by a large factor $R=\Dil(g_c^*m_Y)$, under the new metric on
  $Z \times X$, $\Dil(g_c^*m_Y) \leq 1$.  Now we apply the shadowing principle
  for maps $Z \times X \to Y$ relative to $\partial Z \times X \cup Z \times A$
  with $f=g_c$ and $\ph=g_c^*m_Y$.  This gives us a map homotopic to $g_c$
  (relative to $\partial Z \times X \cup Z \times A$) whose Lipschitz constant
  in the $X$ direction is bounded by $C(Y, n+k)(L+1)$.

  The statement about infinite bars is given by the same proof with
  $\partial Z=\emptyset$.
\end{proof}

\subsection{Nullhomotopic maps to spaces with positive weights}

\begin{thm} \label{thm:ph-pw}
  Let $Y$ be a compact, simply connected Riemannian manifold with boundary
  whose rational homotopy type has positive weights, and let $X$ be a finite
  simplicial complex equipped with a standard simplexwise metric.  Denote the
  component containing the constant maps in $\Lip(X,Y)$ by $\Lip(X,Y)_0$.  Then
  there is a constant $C(Y,\dim X,n)$ such that:
  \begin{enumerate}[(i)]
  \item Every infinite bar in $PH_n(\Lip(X,Y)_0,\log_+\Lip)$ has birth time
    below $C(Y,\dim X,n)$.
  \item Every finite bar in $PH_n(\Lip(X,Y)_0,\log_+\Lip)$ has length bounded
    by $C(Y,\dim X,n)$.
  \end{enumerate}
\end{thm}

\begin{proof} % reorganized version
%  Denote the subspace of $L$-Lipschitz maps in $\Lip(X,Y)_0$ by $\Lip_L(X,Y)_0$.
  The proof follows the same outline as that of Theorem \ref{thm:ph-NR}, by
  inductive application of the following main lemma:
  \begin{lem} \label{lem:push-disk}
    For $k \geq 1$, there is a constant $C_k=C(Y,\dim X,k)$ such that any map
    \[f:(D^k,\partial D^k) \to (\Lip(X,Y)_0,\Lip_L(X,Y)_0)\]
    is homotopic relative to $\partial D^k$ into $\Lip_{C_k(L+1)}(X,Y)_0$.
  \end{lem}
  Assuming this lemma, let $(Z,\partial Z)$ be an $m$-pseudomanifold with
  boundary, and consider a simplicial chain
  \[(c,\partial c):(Z,\partial Z) \to (\Lip(X,Y)_0,\Lip_L(X,Y)_0).\]
  We apply Lemma \ref{lem:push-disk} inductively to the simplices of this
  chain.  At the $k$th step of the induction, we create a map $c_k$ which maps
  the $k$-simplices of $Z$ into $\Lip_{C_1 \cdots C_k(L+1)}(X,Y)_0$.

  To build $c_0$, we homotope $c$ so that each vertex of $Z$ outside
  $\partial Z$ maps to the constant map, and then extend the homotopy to all
  of $Z$ via the homotopy extension property.  For the inductive step
  homotoping $c_k$ to $c_{k+1}$, we leave the map the same on the $k$-skeleton,
  apply Lemma \ref{lem:push-disk} to the $(k+1)$-simplices, and extend to
  higher skeleta by the homotopy extension property.  After the $n$th step we
  get the desired map.

  If $\partial Z$ is empty, we can start the induction with $L=0$ and by sending
  all the $0$-simplices to the constant map.  This gives a map into
  $\Lip_{C_1 \cdots C_n}(X,Y)_0$.
\end{proof}
It remains to prove the lemma.
\begin{proof}[{Proof of Lemma \ref{lem:push-disk}}]
  This is a combination of two auxiliary lemmas.  We start with the following
  weaker result which is essentially proved in \cite{PCDF}:
  \begin{lem} \label{lem:nullh-pw}
    For $k \geq 1$, there is a constant $C=C(Y,k+\dim X)$ such that any
    nullhomotopic map $g:S^{k-1} \to \Lip_L(X,Y)_0$ has an extension
    $\tilde g:S^{k-1} \times [0,1] \to \Lip_{C(L+1)}(X,Y)_0$ so that
    $\tilde g|_{t=0}=g$ and $\tilde g|_{t=1}$ sends every point to a constant
    map to a base point.
  \end{lem}
  \begin{proof}
    Such a map can be thought of alternatively as a map
    $G:S^{k-1} \times X \to Y$.  By assumption it is Lipschitz on fibers over
    points in $S^{k-1}$, but by an arbitrarily small perturbation we can make
    $G$ Lipschitz in the sphere direction as well, without control on the
    Lipschitz constant.  After putting a round metric of sufficiently large
    diameter on $S^{k-1}$, we can even assume that $G$ is $(L+\epsi)$-Lipschitz,
    for arbitrarily small $\epsi>0$.  Then \cite[Theorem~5.6]{PCDF},
    $G|_{\partial\Delta^k}$ has a Lipschitz nullhomotopy through
    $C(Y,k+\dim X)(L+1)$-Lipschitz maps.  We can reinterpret this nullhomotopy
    as a map $S^{k-1} \times [0,1] \to \Lip_{C(L+1)}(X,Y)$.
  \end{proof}
  Since the map $\tilde g$ is constant on $S^{k-1} \times \{1\}$, we may as well
  interpret it as a map from the $k$-disk.  The reason this is insufficient is
  that we don't have control over the relative homotopy class of the resulting
  extension.  The obstruction to building a relative homotopy between the
  extension we started with and this controlled one is an element
  $\alpha \in \pi_k(\Map(X,Y)_0)$.  We resolve the issue by finding a
  controlled representative of this element:
  \begin{lem} \label{lem:pi_k(maps)}
    We can represent every element of $\pi_k(\Map(X,Y)_0)$ by a map
    $S^k \times X \to Y$ whose Lipschitz constant in the $X$ direction is
    bounded by some $C(Y,\dim X,k)$.  More generally, every element of
    $\pi_k(\Map(X,Y)_f)$ for a fixed $f$ has a representative whose Lipschitz
    constant in the $X$ direction is bounded by $C(Y,\dim X,k)(\Lip f+1)$.
  \end{lem}
  \begin{proof}
    That this is true for a constant $C(Y,X,k)$ follows easily from the finite
    generation of $\pi_k(\Map(X,Y))$.  To show that the constant only depends on
    the dimension of $X$, we use the shadowing principle (although it is
    possible that one could apply a more elementary obstruction-theoretic
    method).

    Write $X \xrightarrow{\iota_X} S^k \times X \xrightarrow{\pi_X} X$ for the
    inclusion of the fiber at the base point and the projection, respectively.
    Rational obstruction theory \cite[Prop.~10.5]{GrMo} shows that for any map
    $F:S^k \times X \to Y$ representing an element of $\pi_k(\Map(X,Y)_f)$, we
    can complete the diagram
    \[\xymatrix{
        \mathcal M_Y \ar@{-->}[r]^-\Phi \ar[d]_{m_Y} & \Lambda(e^{(k)})/(e^2)  \otimes \Omega^*X \ar@{^(->}[d]_{d\vol_{S^k} \otimes \id} \ar@{->>}[rd] \\
        \Omega^*Y \ar[r]^-{F^*} & \Omega^*(S^k \times X) \ar[r]^-{\iota_X^*} & \Omega^*X
      }\]
    up to an algebraic homotopy $\mathcal M_Y \to \Omega^*(S^k \times X)
    \otimes \Lambda(t,dt)$ whose composition with $\iota_X^*$ is constant.  In
    particular, for indecomposables $v \in \mathcal M_Y$, we can define
    $\ph(v)$ by
    \[\Phi(v)=\pi_X^*f^*m_Y(v)+e \otimes \ph(v).\]
    Now fix $R$ sufficiently large so that for every indecomposable $v$,
    \[\lVert {d\vol_{S^k}} \wedge \ph(v) \rVert_{\text{op}} \leq \lVert f^*m_Y(v) \rVert_{\text{op}}\]
    with respect to the product metric on $S^k_R \times X$, where $S^k_R$ is a
    round sphere of radius $R$.  Applying the shadowing principle to the
    diagram above gives us a map $S^k_R \times X \to Y$ homotopic to $F$ with
    Lipschitz constant bounded by $C(Y,\dim X,k)(\Lip f+1)$.
  \end{proof}
  Now we construct an extension of $f|_{S^{k-1}}$ to the disk by mapping an outer
  annulus via the nullhomotopy $\tilde g$ obtained in Lemma \ref{lem:nullh-pw}
  and an inner disk via the controlled representative of the obstruction class
  as given by Lemma \ref{lem:pi_k(maps)}.  This extension is homotopic to the
  original map $f$.
\end{proof}

\subsection{Maps to rational H-spaces}
\begin{thm} \label{thm:ph-1step}
  Let $Y$ be a compact, simply connected Riemannian manifold with boundary
  which is rationally equivalent to a product of Eilenberg--MacLane spaces, and
  let $X$ be a finite simplicial complex equipped with a standard simplexwise
  metric.  Then there is a constant $C(Y,\dim X,n)$ such that:
  \begin{enumerate}[(i)]
  \item Every infinite bar in $PH_n(\Lip(X,Y)_f,\log_+\Lip)$ has birth time in
    \[\log_+\Lip f,\log_+\Lip f+C(Y,\dim X,n).\]
  \item Every finite bar in $PH_n(\Lip(X,Y)_0,\log_+\Lip)$ has length bounded
    by $C(Y,\dim X,n)$.
  \end{enumerate}
\end{thm}
A simply connected space is rationally equivalent to a product of Eilenberg--MacLane spaces if and only if it is rationally equivalent to an H-space.  Such spaces include odd-dimensional spheres, connected Lie groups, as well as classifying spaces of connected Lie groups \cite[Prop.~15.15]{FHT}.

The proof of this theorem is identical to that of Theorem \ref{thm:ph-pw}, except that Lemma \ref{lem:nullh-pw} is replaced by the following fact:
\begin{lem}
  For $k \geq 1$, there is a constant $C(Y,k+\dim X)$ such that any map
  $g:S^{k-1} \to \Lip_L(X,Y)$ has an extension
  $\tilde g:S^{k-1} \times [0,1] \to \Lip_{C(L+1)}(X,Y)$ so that
  $\tilde g|_{t=0}=g$ and $\tilde g|_{t=1}$ sends every point to the same map.
\end{lem}
Just as Lemma \ref{lem:nullh-pw} follows from \cite[Theorem 5.6]{PCDF}, this lemma follows from \cite[Theorem B]{CDMW}, or rather its more precise restatement in \S4 of that paper.

\section{Examples with bars of unbounded length} \label{S:example}

\begin{thm} \label{example}
  Let $X=\CC P^2 \times S^3$ or $S^2 \times S^2 \times (S^3 \vee S^3)$.  Then
  there is a sequence of $L \to \infty$ and pairs of $L$-Lipschitz maps
  \[f_L,g_L:X \to S^3 \vee S^3\]
  which are homotopic, but not through maps of Lipschitz constant $o(L^{4/3})$.
\end{thm}

\subsection{Notation and conventions} \label{S:notation}
We will implicitly treat our spaces as CW complexes with the simplest possible CW structure.  That is, $S^n$ consists of a $0$-cell and an $n$-cell, $\CC P^2$ consists of a $0$-cell, a $2$-cell and a $4$-cell, and products are equipped with the product cell structure.  In particular, $S^2$ is unambiguously a subcomplex of $\CC P^2$ and factors are subcomplexes of the product.

Given a map $f:S^n \to Y$ and an integer $D$, we denote by $Df$ the composition
\[S^n \xrightarrow{g_D} S^n \xrightarrow{f} Y,\]
where $g_D$ is a map of degree $D$ and Lipschitz constant $O(\lvert D \rvert^{1/n})$.  Therefore $\Lip(Df)=O(\lvert D \rvert^{1/n})\Lip f$.  Similarly, given maps $f:S^n \to Y$, $g:S^m \to Y$, we denote by $[f,g]$ the specific map
\[S^{m+n-1} \xrightarrow{\text{attaching map of the top cell of }S^m \times S^n} S^n \vee S^m \xrightarrow{f \vee g} Y\]
representing the Whitehead product of $f$ and $g$, so that $\Lip [f,g]=\text{const} \cdot \max\{\Lip f,\Lip g\}$.

\subsection{Informal explanation}
We begin by describing for the case of $X=\CC P^2 \times S^3$ the ordinary obstruction theory that gives some intuition for the source of these examples.  We write $Y=S^3 \vee S^3$, with $\iota_1,\iota_2:S^3 \to Y$ representing the inclusions of the two spheres.

As a warmup, consider the simpler domain space $S^2 \times S^3$.  One way of producing maps $S^2 \times S^3 \to Y$ is to take the projection to $S^3$ followed by a map in the homotopy class $d_1[\iota_1]+d_2[\iota_2]$.  One can also produce other maps by altering the map on the $5$-cell of $S^2 \times S^3$ by an element of
\[\pi_5(S^3 \vee S^3) \cong \ZZ \oplus \text{finite},\]
where the $\ZZ$ factor is generated by the Whitehead product $[\iota_1,\iota_2]$; let $h$ be the coefficient of $[\iota_1,\iota_2]$.  Obstruction theory tells us that this gives a complete list of homotopy classes, and therefore $d_1$, $d_2$, and $h$ determine the homotopy class up to finite ambiguity.  In fact, relative to the $S^2$ factor, which all of these maps send to the basepoint, these are all distinct homotopy classes.

On the other hand, in the non-relative setting, if one of the $d_i$ is nonzero, then $h$ is only well-defined modulo $\operatorname{gcd}(d_1,d_2)$.  (A similar observation was made in \cite[\S5]{CMW}.)  In fact, while a homotopy cannot change the $d_i$, a homotopy whose restriction to $S^2 \times [0,1]$ winds $a_i$ times around the $i$th sphere always satisfies
\begin{equation} \label{h vs d}
  h_{\text{end}}-h_{\text{start}}=a_1d_2-d_1a_2.
\end{equation}
This can be seen geometrically by factoring the homotopy through a quotient map on the ends, $S^2 \times S^3 \times \{0,1\}$, which pinches off a $5$-sphere from the top cell of $S^2 \times S^3$, then projects the rest to $S^3$ (see Figure \ref{figure}).
\begin{figure} 
	\centering
  \begin{tikzpicture}
%    \draw[gray, very thin, step=0.5] (-2,-2) grid (7.6,1);
    \useasboundingbox (-2,-2.5) rectangle (7.6,1);
    \draw (1,1) .. controls (-2,.5) and (-.5,-1.5) .. (1,1);
    \node(S4) at (.1,.4){$S^2$};
    \draw (1,1) .. controls (.8,-2.5) and (4.5,1.2) .. (1,1);
    \node(S4) at (1.8,.3){$S^3$};
    \draw (-.65,-.18) .. controls (0,-.7) and (-.1,-1) .. (-.7,-.8);
    \draw (-.7,-.8) .. controls (-3.5,.4) and (-1.6,-3.5) .. (-.6,-1.5);
    \draw (1.8,-.5) .. controls (.2,-.8) and (0,-.8) .. (-.6,-1.5);
    \draw[->] (2,1) .. controls (2.5,1.2) and (3,1.2) .. (3.7,.8);
    \draw[->] (-.7,-1.8) .. controls (1,-2.1) and (2,-1.9) .. (3.1,-1.6);
    \draw (5,-.5) .. controls (5.2,3) and (1.5,-.5) .. (5,-.5);
    \node(S4) at (4.2,.3){$S^3$};
    \draw (5,-.5) .. controls (4.6,-4) and (1,-.5) .. (5,-.5);
    \node(S4) at (4,-1.3){$S^5$};
    \draw[->] (4.95,.65) .. controls (5.5,0.9) and (6.15,0.9) .. (6.65,.65)
      node[pos=0.5, anchor=south]{degree $d$};
    \draw[->] (4.8,-1.5) .. controls (6,-1.8) and (6,-.65) .. (6.7,-.65)
      node[pos=0.3, anchor=north, align=center]{Whitehead\\coefficient $h$};
    \draw (7.5,0.3) circle [radius=.8] node{$S^3$};
    \draw (7.5,-1.3) circle [radius=.8] node{$S^3$};
  \end{tikzpicture}
  
  \caption{
    Construct maps $u_{d,h}:S^2 \times S^3 \to Y$ by ``budding off'' a small
    sphere and then projecting the rest onto the $S^3$ factor.
  } \label{figure}
\end{figure}
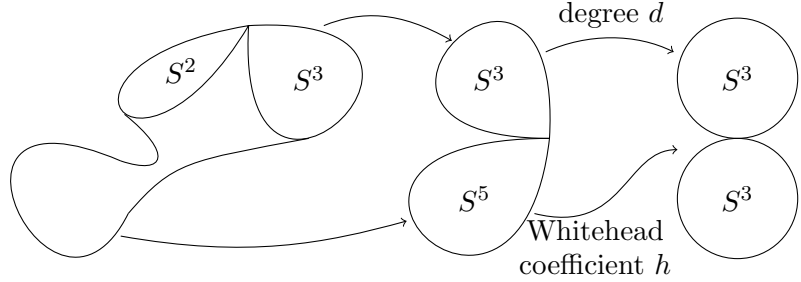
On the whole interval, this extends to a quotient map
\begin{equation} \label{eq:quotient}
  S^2 \times S^3 \times I \to S^3 \times S^3 \setminus \text{two open balls},
\end{equation}
where the first $S^3$ is the quotient of $S^2 \times [0,1]$.  The sum of the boundary maps of the two balls is homotopic to the Whitehead product of the two $S^3$'s, which implies \eqref{h vs d}.

In particular, let 
\[u_{d,h}:S^2 \times S^3 \to S^3 \vee S^3\]
be a map which differs from $d[\iota_1] \circ \pi_{S^3}$ by Whitehead coefficient $h$.  Then the above construction gives:
\begin{prop} \label{prop:easy-htpy}
  When $h-h' \equiv 0 \mod d$, then $u_{d,h}$ and $u_{d,h'}$ are homotopic.
\end{prop}

Now consider instead the domain space $X=\CC P^2 \times S^3$, which naturally
contains $S^2 \times S^3$ as a subcomplex.  First notice:
\begin{prop} \label{prop:ext}
  When $h$ is even, $u_{d,h}$ extends to a map $\tilde u_{d,h}:X \to S^3 \vee S^3$.
\end{prop}
\begin{proof}
  Consider the space $X'=(\CC P^2 \times S^3)/(\CC P^2 \times {*})$, and let
  $q:X \to X'$ be the quotient map. It's not hard to see that $X'$ is homotopy
  equivalent to the third suspension $S^3(\CC P^2 \sqcup {*})$.  This complex
  consists of cells in dimensions $3$, $5$ and $7$, with attaching maps
  obtained by stabilizing those of the original complex.  So the attaching map
  of the $5$-cell is trivial and the attaching map of the $7$-cell is the
  nontrivial element of $\pi_6(S^5) \cong \ZZ/2\ZZ$.  Composing this element
  with a map $S^5 \to S^5$ of degree $2$ trivializes it (since before
  suspending, it multiplies the Hopf invariant by $4$.)  Therefore there is a
  map $g:X' \to S^3 \vee S^5 \vee S^7$ which has degree $1$ on the $3$- and
  $7$-cells and degree $2$ on the $5$-cell.

  Then $u_{d,h}$ extends (up to homotopy) to a composition
  \[X \xrightarrow{q} X' \xrightarrow{g} S^3 \vee S^5 \vee S^7 \xrightarrow{d[\iota_1] \vee \frac{h}{2}[\iota_1,\iota_2] \vee 0} S^3 \vee S^3. \qedhere\]
\end{proof}
As before, we can modify this map on the $7$-cell by elements of $\pi_7(S^3 \vee S^3)$, whose rational part is generated by the two triple Whitehead products $[\iota_1,[\iota_1,\iota_2]]$ and $[\iota_2,[\iota_1,\iota_2]]$.

It turns out, although this is harder to see geometrically, the homotopy described above between $u_{1,0}$ and $u_{1,h}$ also extends over the $7$-cell in many cases, but the restrictions to the $7$-cell of the maps on either side of the resulting homotopy differ by $h^2[\iota_2,[\iota_1,\iota_2]]$.  It follows that the simplest lifts $\tilde u_{1,-h}$ and $\tilde u_{1,h}$ are homotopic (at least up to a torsion obstruction), but any homotopy between them passes through a map with complicated behavior on the $7$-cell, and hence a large Lipschitz constant.  Below, we use rational homotopy theory to prove this directly (without making precise the heuristic involving $u_{1,0}$).

The example of $S^2 \times S^2 \times (S^3 \vee S^3)$ is more complicated, but again uses quadratic behavior of the homotopy class on the $7$-cells.  This time there are two $7$-cells and two quadratic functions in two variables; we show that at some point during the homotopy, one of these functions must attain a large value.

\subsection{Algebraic model}
We now begin the formal proof of Theorem \ref{example} for $X=\CC P^2 \times S^3$.  At the end of the section we will describe the analogous proof for $X=S^2 \times S^2 \times (S^3 \vee S^3)$.

We start by building simplified algebraic models of the maps and homotopies between them and showing that these algebraic homotopies must be large.  Afterwards we will deduce the results for Lipschitz constants of genuine maps and homotopies between them.

Write $Y=S^3 \vee S^3$, write
\begin{align*}
  \mathcal M_Y^* &= (\Lambda(a_1^{(3)},a_2^{(3)},b^{(5)},c_1^{(7)},c_2^{(7)},\ldots), da_i=0, db=a_1a_2, dc_i=ba_i,\ldots) \\
  H^*(X) &= \Lambda(x^{(2)},y^{(3)})/(x^3),
\end{align*}
and consider the homomorphisms $\ph_L,\psi_L:\mathcal M_Y \to H^*(X)$ given by
\begin{align*}
  \ph_L(a_1) &= y & \psi_L(a_1) &= y \\
  \ph_L(a_2) &= 0 & \psi_L(a_2) &= 0 \\
  \ph_L(b) &= -L^6yx & \psi_L(b) &= L^6yx \\
  \ph_L(c_i) &= 0 & \psi_L(c_i) &= 0.
\end{align*}
These two homomorphisms are homotopic via the homotopy $\eta_L:\mathcal M_Y^* \to H^*(X) \otimes \Omega^*[0,1]$ given by:
\begin{align*}
  \eta_L(a_1) &= y & \eta_L(a_2) &= -2L^6xdt & \eta_L(b) &= L^6yx(2t-1) \\
  \eta_L(c_1) &= 0 & \eta_L(c_2) &= 2L^{12}yx^2t(1-t).
\end{align*}
On the other hand:

\begin{lem} \label{lem:alg-dil}
  Every homotopy $\eta$ between $\ph_L$ and $\psi_L$ satisfies $\Dil(\eta)=\Omega(L^{12/7})$.
\end{lem}
\begin{proof}
  We inductively explore the ways of defining such a homotopy, noting that the
  space of possibilities is rather small.  First, we must have
  \[\eta(a_1)=y-xd\alpha(t) \qquad \eta(a_2)=-xd\beta(t),\]
  where $\alpha,\beta:[0,1] \to \RR$ are Lipschitz functions.  The relation
  $\eta d=d\eta$ forces us to set
  \[\eta(b)=yx\beta(t)+x^2d\gamma(t),\]
  where $\gamma:[0,1] \to \RR$ is again Lipschitz and the definitions of
  $\ph_L$ and $\psi_L$ mean that $\beta(0)=-L^6$ and $\beta(1)=L^6$.  Finally,
  again using $\eta d=d\eta$ and the definitions of the functions, we get
  \[\eta(c_2)=yx^2(L^{12}-\frac{1}{2}\beta^2(t)).\]
  By the intermediate value theorem, $\beta(t)=0$ for some $t \in (0,1)$, so
  $\lVert \eta(c_2) \rVert_\infty=\Omega(L^{12})$.  Since $\deg c_2=7$, we have
  $\Dil(\eta)=\Omega(L^{12/7})$.
\end{proof}

\subsection{Construction of homotopic maps}
Define $X'=(\CC P^2 \times S^3)/(\CC P^2 \times {*})$ and fix maps
\[X \xrightarrow{q} X' \xrightarrow{g} S^3 \vee S^5 \vee S^7\]
as in the proof of Proposition \ref{prop:ext}.  Now, using the conventions of \S\ref{S:notation}, we define the maps
\begin{align*}
  f_L: \quad X &\xrightarrow{q} X' \xrightarrow{g} S^3 \vee S^5 \vee S^7 \xrightarrow{\iota_1 \vee [-L^3\iota_1,L^3\iota_2] \vee 0} S^3 \vee S^3, \\
  g_L: \quad X &\xrightarrow{q} X' \xrightarrow{g} S^3 \vee S^5 \vee S^7 \xrightarrow{\iota_1 \vee [L^3\iota_1,L^3\iota_2] \vee \alpha_L} S^3 \vee S^3,
\end{align*}
where $\alpha_L \in \pi_7(S^3 \vee S^3)$ is an $O(1)$-Lipschitz representative of a torsion element to be determined.  By construction, $f_L$ and $g_L$ are $O(L)$-Lipschitz.

\begin{lem} \label{lem:htpy-of-ends}
  There are algebraic homotopies of dilatation $O(L^{6/5})$ from $f_L^*m_Y$ and
  $g_L^*m_Y$ to $\rho_X \circ \ph_L$ and $\rho_X \circ \psi_L$, respectively,
  where $\rho_X:H^*(X) \to \Omega^*(X)$ is a way of realizing the cohomology
  algebra using differential forms.
\end{lem}
\begin{proof}
  We give the proof for $g_L$ and $\psi_L$; the proof for $f_L$ and $\ph_L$ is
  identical.  Write $Y=S^3 \vee S^3$ and $v_L=[L^3\iota_1,L^3\iota_2]$, and
  consider the diagram
  \[\xymatrixcolsep{4pc}\xymatrix{     
      \mathcal M_Y \ar[d]^{m_Y} \ar[r]^-{\nu_L} & H^*(S^3) \oplus H^*(S^5) \oplus H^*(S^7) \ar[d]^{\rho_\vee} \ar[r]^-{q^*g^*} & H^*(X) \ar[d]^{\rho_X} \\
      \Omega^*Y \ar[r]^-{(\iota_1^*,v_L^*,\alpha_L^*)} & \Omega^*S^3 \oplus \Omega^*S^5 \oplus \Omega^*S^7 \ar[r]^-{q^*g^*} & \Omega^*X.
    }\]
  Here $\rho_\vee$ takes each generator to the corresponding fundamental class,
  and $\nu_L$ sends
  \[a_1 \mapsto [S^3], \qquad a_2 \mapsto 0, \qquad b \mapsto L^6[S^5].\]
  We will show that the three paths through this diagram are homotopic via
  homotopies of dilatation $O(L^{6/5})$.  Note that $q^*g^*\nu_L=\psi_L$.

  We first handle the left square.  By applying Lemma \ref{lem:formalize}(i)
  separately to $\iota_1$, $v_L$, and $\alpha_L$, we get an algebraic homotopy
  $\Phi$ of dilatation $O(L^{6/5})$ from $(\iota_1^*,v_L^*,\alpha_L^*)^*m_Y$
  to $\rho_\vee \circ \nu_L$.  (We have $b([\iota_1,\iota_2])=1$ by
  \cite[Theorem 6.1]{AArk}, relating the Whitehead product to the quadratic
  part of the differential in the minimal model.)  Composing with
  $(q^*g^* \otimes \id)$, we get a homotopy between the lower and middle path.

  The middle path $q^*g^*\rho_\vee\nu_L$ sends
  \[a_1 \mapsto q^*g^*d\vol_{S^3}, \qquad b \mapsto L^6q^*g^*d\vol_{S^5},\]
  and all other generators to $0$.  Similarly, the upper path $\rho_X\psi_L$
  sends
  \[a_1 \mapsto \rho_X(y), \qquad b \mapsto L^6\rho_X(yx),\]
  and all other generators to $0$.  Choose closed forms
  $\omega_3 \in \Omega^3(X \times [0,1])$ interpolating between
  $q^*g^*d\vol_{S^3}$ and $\rho_X(y)$ and
  $\omega_5 \in \Omega^5(X \times [0,1])$ interpolating between
  $q^*g^*d\vol_{S^5}$ and $\rho_X(yx)$.  Then the map sending
  $a_1 \mapsto \omega_3$, $b \mapsto L^6\omega_5$ and all other generators to
  $0$ is a homotopy of dilatation $O(L^{6/5})$.
\end{proof}

\begin{lem} \label{lem:homotopic}
  For an appropriate choice of $\alpha_L$, $f_L$ and $g_L$ are homotopic.
\end{lem}
\begin{proof}
  Clearly, $f_L|_{S^2 \times S^3}$ is homotopic to the map $u_{1,-L^6}$
  constructed earlier, and similarly, $g_L|_{S^2 \times S^3}$ is homotopic to
  $u_{1,L^6}$.  Thus by Proposition \ref{prop:easy-htpy}, we have a homotopy
  between these two restrictions which maps $S^2 \times [0,1]$ via
  $2L^6[\iota_2]$.  In other words, so far we have a map defined on
  \[(X \times \{0,1\}) \cup (S^2 \times S^3 \times [0,1]).\]
  Moreover, this map takes the boundary of the extra $5$-cell of
  $\CC P^2 \times [0,1]$ to $S^3 \vee S^3$ via
  \[2L^6[\iota_2] \circ (S\operatorname{Hopf}):S^4 \to S^3 \vee S^3,\]
  which is nullhomotopic in the second $S^3$, by the same argument as in the
  proof of Proposition \ref{prop:ext}.  Thus the map extends to this $5$-cell,
  yielding a map
  \[F:(X \times [0,1])^{(7)} \to Y\]
  which takes $\CC P^2 \times [0,1]$ to the second $S^3$.

  It remains to understand the obstruction in $\pi_7(S^3 \vee S^3)$ to
  extending this map over the $8$-cell of $\CC P^2 \times S^3 \times [0,1]$.
  Assuming for now that $\alpha_L=0$, $f_L$ and $g_L$ both factor through
  $u:X \to S^3 \vee S^5$, and therefore $F$ factors through a map
  $\overline{F}:Z^{(7)} \to Y$, where
  \[Z=\frac{X \times [0,1] \sqcup (S^3 \vee S^5) \times \{0,1\}}{(x,t) \sim (u(x),t)}.\]
  Note that the quotient map $X \times [0,1] \to Z$ extends that of
  \eqref{eq:quotient}, and in particular
  \[Z^{(7)}=[(S^3 \times S^3) \setminus \text{two open balls}] \cup e^5\]
  where the $5$-cell represents the suspension of the $4$-cell of $\CC P^2$.
  It follows that
  \[Z^{(7)} \simeq_{\QQ} S^3 \vee S^3 \vee S^5 \vee S^5,\]
  and the rational homotopy type of $\overline F$ is determined by the images
  of these four spheres.  These are:
  \begin{itemize}
  \item $[\iota_1]$, for the $S^3$ factor of $X$;
  \item $2L^6[\iota_2]$, for the quotient of $S^2 \times [0,1]$;
  \item $-L^6[\iota_1,\iota_2]$, for the $S^5$ at time $0$;
  \item $0$, for the suspension of the $4$-cell.
  \end{itemize}

  To understand the obstruction to extending $\overline F$ over the $8$-cell,
  we construct a rational model.  Consider the DGA
  \[\mathcal A=\{(w,w'_0,w'_1) \in H^*(X) \otimes \Lambda(t,dt) \oplus H^*(S^3 \vee S^5)^2\;: w|_{t=0}=u^*(w'_0)\text{ and }w|_{t=1}=u^*(w'_1)\}\]
  representing the rational homotopy type of $Z$.  (Here we are using the
  contravariance of the functor from spaces to DGAs: a quotient of a disjoint
  union becomes a subalgebra of a direct sum.)  Since $S^3 \vee S^5$ and $X$
  are both scalable, there is an evaluation map
  $\operatorname{ev}:\mathcal A \to \Omega^*Z$.  Now define a homomorphism
  \[\overline\eta_L=(\eta_L,\nu_L,\nu'_L):\mathcal M_Y \to \mathcal A\]
  by
  \begin{align*}
    a_1 &\mapsto (y, [S^3], [S^3]) &
    a_2 &\mapsto (-2L^6xdt, 0, 0) \\
    b &\mapsto \mathrlap{(L^6yx(2t-1),-L^6[S^5],L^6[S^5])} \\
    c_1 &\mapsto (0,0,0) &
    c_2 &\mapsto (2L^{12}yx^2t(1-t),0,0).
  \end{align*}
  To see that $(\operatorname{ev}\circ\overline\eta_L)|_{Z^{(7)}}$ is homotopic
  to $\overline F^*m_Y$, notice that the projections of $\overline\eta_L(a_1)$,
  $\overline\eta_L(a_2)$ and $\overline\eta_L(b)$ to $H^*(Z)$ are dual to the
  map described above.

  It follows that there is no obstruction to extending $\overline F^*m_Y$ over
  the $8$-cell of $Z$, and so any obstruction to extending $\overline F$ must
  be a torsion element of $\pi_7(S^3 \vee S^3)$.  We can take this element to
  be $\alpha_L$.
\end{proof}

Finally, we complete the proof of Theorem \ref{example}:
\begin{lem} \label{lem:finish}
  Every homotopy between $f_L$ and $g_L$ goes through maps of Lipschitz constant
  $\Omega(L^{4/3})$.
\end{lem}
\begin{proof}
  By Lemma \ref{lem:alg-dil}, any homotopy between $\ph_L$ and $\psi_L$ has
  dilatation $\Omega(L^{12/7})$.  By Lemma \ref{lem:formalize}(ii), from an
  $L'$-Lipschitz homotopy between $f_L$ and $g_L$ we can obtain a homotopy
  between $\ph_L$ and $\psi_L$ of dilatation $O((L')^{9/7})$.  Therefore, any
  homotopy $h:X \times [0,T] \to Y$ between $f_L$ and $g_L$ has Lipschitz
  constant $\Omega(L^{\frac{12}{7}/\frac{9}{7}})=\Omega(L^{4/3})$.  Since we can
  stretch the time direction arbitrarily, this Lipschitz constant must be that
  of a time-slice of the homotopy.
\end{proof}

\subsection{Proof for $X=S^2 \times S^2 \times (S^3 \vee S^3)$}
The proof in this case follows the same outline.  As before, we start with an algebraic argument.  Write $X=S^2 \times S^2 \times (S^3 \vee S^3)$ and
\[H^*(X)=\Lambda(x_1,x_2,y_1,y_2)/(y_1y_2,x_1^2,x_2^2),\]
and consider the homomorphisms $\ph_L,\psi_L:\mathcal M_Y \to H^*(X)$ given by
\begin{align*}
  \ph_L(a_1) &= y_1-y_2 & \psi_L(a_1) &= y_1-y_2 \\
  \ph_L(a_2) &= 0 & \psi_L(a_2) &= 0 \\
  \ph_L(b) &= L^6(x_1y_1+x_2y_2) & \psi_L(b) &= L^6(x_1y_2+x_2y_1) \\
  \ph_L(c_i) &= 0 & \psi_L(c_i) &= 0.
\end{align*}
These two homomorphisms are homotopic via the homotopy $\eta_L:\mathcal M_Y^* \to H^*(X) \otimes \Omega^*[0,1]$ given by:
\begin{align*}
  \eta_L(a_1) &= y_1-y_2 & \eta_L(a_2) &= L^6(x_1-x_2)dt \\
  \eta_L(b) &= \mathrlap{L^6[(x_1y_1+x_2y_2)(1-t)+(x_1y_2+x_2y_1)t]} \\
  \eta_L(c_1) &= 0 & \eta_L(c_2) &= L^{12}x_1x_2(y_1-y_2)t(1-t). % check signs
\end{align*}
On the other hand:

\begin{lem} \label{lem:alg-dil2}
  Every homotopy between $\ph_L$ and $\psi_L$ satisfies $\Dil(\eta)=\Omega(L^{12/7})$.
\end{lem}
\begin{proof}
  Any such homotopy must have
  \begin{align*}
    \eta(a_1) &= y_1-y_2+x_1d\alpha_1(t)+x_2d\alpha_2(t) \\
    \eta(a_2) &= L^6(x_1d\beta_1(t)+x_2d\beta_2(t)),
  \end{align*}
  where $\alpha_i,\beta_i:[0,1] \to \RR$ are Lipschitz functions.  Then,
  normalizing so that $\beta_1(0)=\beta_2(0)=0$, the relation $\eta d=d\eta$
  forces us to have
  \[\eta(b) = L^6\left[(1-\beta_1(t))x_1y_1 + (1-\beta_2(t))x_2y_2 + \beta_1(t)x_1y_2 + \beta_2(t)x_2y_1\right] + x_1x_2d\gamma(t),\]
  from which we see that $\beta_1(1)=\beta_2(1)=1$.  Finally, again using
  $\eta d=d\eta$, we get
  \[\eta(c_2) = L^{12}\left[\beta_2(t)(1-\beta_1(t))y_1 - \beta_1(t)(1-\beta_2(t))y_2\right] x_1x_2.\]

  Now suppose that for all $t$,
  \[\lvert\beta_1(t)(1-\beta_2(t))\rvert + \lvert\beta_2(t)(1-\beta_1(t))\rvert < \frac{1}{6}.\]
  Then by the reverse triangle inequality, for all $t$,
  \[\lvert\beta_1(t)-\beta_2(t)\rvert<1/6.\]
  But by the intermediate value theorem, there is a point where
  $\beta_1(t)=1/2$.  At this point, $\beta_2(t)<2/3$, and therefore
  \[\beta_1(t)(1-\beta_2(t))>1/6,\]
  a contradiction.  Therefore, $\Dil(\eta)=\Omega(L^{12/7})$.
\end{proof}

Now we define the maps $f_L$ and $g_L$.  Define
\[s_{h^2,d}:S^2 \times S^3 \to S^3 \vee S^3\]
to be the composition
\[S^2 \times S^3 \xrightarrow{f_0} S^5 \vee S^3 \xrightarrow{[i_1r_h,i_2r_h] \vee (i_1r_d)} S^3 \vee S^3,\]
where
\begin{itemize}
\item $f_0$ is a map which pinches off a sphere and projects the remainder of
  the space onto the $S^3$ factor.
\item $i_1$ and $i_2$ are the inclusions of the two spheres in the wedge.
\item $r_p$ is an $O(p^{1/3})$-Lipschitz self-map of $S^3$ of degree $p$.
\end{itemize}
Define $\pi_1,\pi_2:S^2 \times S^2 \times S^3 \to S^2 \times S^3$ to be the projections that project onto the first and second copy of $S^2$, respectively.  Then we define
\[f_L(x)=\begin{cases}
    s_{L^6,1} \circ \pi_1(x) & x \in S^2 \times S^2 \times S^3_1 \\
    s_{L^6,-1} \circ \pi_2(x) & x \in S^2 \times S^2 \times S^3_2,
  \end{cases}\]
which is well-defined since both maps take the intersection of the two
subdomains to the base point.  Similarly, we define
\[g_L(x)=\begin{cases}
    s_{L^6,1} \circ \pi_2(x) & x \in S^2 \times S^2 \times S^3_1 \\
    s_{L^6,-1} \circ \pi_1(x) & x \in S^2 \times S^2 \times S^3_2.
  \end{cases}\]
Clearly these maps are $O(L)$-Lipschitz.

A similar argument to Lemma \ref{lem:homotopic} shows that these maps are homotopic (perhaps after modifying $g_L$ on the two $7$-cells by some torsion elements of $\pi_7(S^3 \vee S^3)$).  A similar argument to Lemma \ref{lem:htpy-of-ends} shows that they have algebraic homotopies of dilatation $O(L^{6/5})$ to $\rho \circ \ph_L$ and $\rho \circ \psi_L$, respectively.  The argument of Lemma \ref{lem:finish} completes the proof for this case.

\section{Optimality of the example} \label{S:optimal}

In this section we show, using related techniques, that the example above is optimal in two different senses.  First, two homotopic $L$-Lipschitz maps from any $7$-complex to $S^3 \vee S^3$ are always homotopic through $O(L^{4/3})$-Lipschitz maps.  Secondly, if $Y$ is a scalable space with nontrivial rational homotopy groups in only two degrees, then all pairs of homotopic $L$-Lipschitz maps to $Y$ are homotopic through $O(L)$-Lipschitz maps; so our example with nontrivial rational obstruction groups in three degrees is in some sense the simplest possible.

We start with the second part:
\begin{thm} \label{thm:2step}
  Let $Y$ be a compact, simply connected Riemannian manifold (with boundary)
  which is scalable and rationally equivalent to the total space of a principal
  fibration whose base and fiber are products of Eilenberg--MacLane spaces, and
  let $X$ be a finite simplicial complex equipped with a standard simplexwise
  metric.  Then every pair of homotopic $L$-Lipschitz maps $f,g:X \to Y$ is
  homotopic through $C(X,Y)(L+1)$-Lipschitz maps.  Moreover, the Lipschitz
  constant in the time direction of the homotopy takes the form
  $C(X,Y)(L+1)^p$, where $p=p(X,Y)$ is some integer exponent.
\end{thm}
In other words, the finite bars of $PH_0(\Lip(X,Y),\log_+\Lip)$ have bounded length in all components.  We note, however, that the constant $C(X,Y)$ depends not only on the local geometry of $X$, but also its global geometry.  This means that Theorem \ref{thm:2step} is insufficient to show, using the technique of Theorem \ref{thm:ph-pw}, the boundedness of finite bars in all degrees.

While the conditions on $Y$ are rather strong, they include all compact-type symmetric spaces: symmetric spaces are scalable, and all homogeneous manifolds satisfy the ``two-stage'' condition \cite[Prop.~15.16]{FHT}.
\begin{proof}
  This proof is similar to the proofs of \cite[Theorems 5.5 and 5.6]{PCDF}.  We
  construct an algebraic homotopy and then apply the argument of
  \cite[Theorem 5.7]{PCDF} to make sure that it lies in the relative homotopy
  class of a genuine homotopy between $f$ and $g$.  We then use the improved
  shadowing principle for scalable spaces (rather than the original shadowing
  principle as in the proofs we imitate) to construct a homotopy with
  controlled Lipschitz constant.

  Scalability gives an embedding $H^*(Y;\RR) \to \Omega^*Y$.  Any minimal model
  $\mathcal M_Y^* \to \Omega^*Y$ lifts through this embedding.  So fix a
  minimal model $m_Y:\mathcal M_Y^* \to \Omega^*Y$ which factors through
  $H^*(Y;\RR)$.  By assumption, $\mathcal M_Y$ decomposes as
  \[\mathcal M_Y=\Lambda(V \oplus W)\]
  where $dV=0$ and $dW \subseteq \Lambda V$, and $H^*(Y;\RR)$ is a quotient of
  $\Lambda V$.  We also fix a linear map $\alpha:H^*(X;\RR) \to \Omega^*X$
  sending cohomology classes to representatives.  (Note this may not be a ring
  homomorphism, which does not exist unless $X$ is also scalable.)

  We start by producing an algebraic homotopy between $f^*m_Y$ and $g^*m_Y$,
  which will in turn consist of two steps.  First we build a homotopy
  \[\Phi_1:\mathcal M_Y^* \to \Omega^*(X \times [0,1])\]
  between $f^*m_Y$ and a homomorphism $\ph$ which coincides with
  $g^*m_Y$ on $V$.  For $v \in V$ of degree $k$, we set
  \[\Phi_1(v)=(1-t^k)f^*m_Y(v)+t^kg^*m_Y(v)+c(v) \otimes dt,\]
  where $c(v)$ is chosen so that $dc(v)=(-1)^{k+1}(g^*m_Y(v)-f^*m_Y(v))$.
  For $v \in W$, we then choose an extension using Prop.~\ref{prop:extHtpy}
  whose operator norm is $O(L^{\deg v+1})$.  Thus $\Dil^U(\Phi_1)=O(L)$, and
  $\ph=\Phi_1|_{t=1}$ sends $W$ to closed forms.

  Since $g^*m_Y|_W=0$, $[\ph|_W] \in \Hom(W,H^*(X;\RR))$ measures the
  obstruction to extending the constant homotopy on $V$ to a homotopy between
  $\ph$ and $g^*m_Y$.  On the other hand, since we know a homotopy between
  these homomorphisms exists, by \cite[Prop.~14.4]{GrMo} $[\ph|_W]$ must be in
  the image of the homomorphism
  \[[V,\Omega^*X \otimes \Lambda(e^{(1)})]_{\alpha f^*} \to \Hom(W,H^*(X;\RR))\]
  sending $[g^*m_Y+\eta \otimes e] \mapsto [\eta d|_W]$.  The domain of this
  homomorphism is isomorphic to $\Hom(V,H^{*-1}(X;\RR))$, and we can choose
  specific representatives using $\alpha$; in this way, choose a preimage
  $g^*m_Y+\alpha\tilde\ph \otimes e$ for $\ph$ (in a way that minimizes
  $\lVert\alpha\tilde\ph\rVert_{\text{op}}$).  Now we define a homomorphism
  $\Phi_2:\mathcal M_Y^* \to \Omega^*(X \times [0,T])$ by
  \begin{align*}
    \Phi_2(v) &= g^*m_Y(v)-dt/T \wedge \alpha\tilde\ph(v) & v &\in V \\
    \Phi_2(w) &= (1-t/T)\ph(w)-dt/T \wedge c(w) & w &\in W,
  \end{align*}
  where $dc(w)=\ph(w)+\int_0^T \Phi_2(dw)$.  Then $\Phi_2$ is a homotopy
  from $\ph$ to $g^*m_Y$, and for
  \[T=\sup_{v \in V} L^{-\deg v}\lVert \alpha\tilde\ph(v) \rVert_{\text{op}},\]
  it has $U$-dilatation $O(L)$.  (By an argument given in the proof of
  \cite[Theorem 5.5(ii)]{PCDF}, this estimate for $T$ is in general
  polynomially bounded in $L$.)

  The concatenation $\Phi$ of $\Phi_1$ and $\Phi_2$ is an algebraic homotopy
  between $f^*m_Y$ and $g^*m_Y$, but it may not lie in the relative homotopy
  class of a genuine homotopy of maps.  Fixing this is the purpose of
  \cite[Theorem 5.7]{PCDF}; we summarize the proof here for convenience.

  Build a homomorphism $\Sigma:\mathcal M_Y^* \to \Omega^*(X \times S^1)$ by a
  circular concatenation of $\Phi$ and $h^*m_Y$, where $h$ is a genuine, but
  uncontrolled, homotopy between $f$ and $g$.  This homomorphism lifts up to
  homotopy to a homomorphism
  \[g^*m_Y+\eta \otimes e:\mathcal M_Y^* \to \Omega^*X \otimes \Lambda(e^{(1)})\]
  representing an element of $\pi_1(Y^X,g) \otimes \RR$.  By
  \cite[Lemma 5.2(i)]{IRMC} and the surrounding discussion, we can find an
  $\eta'$ such that $[g^*m_Y+\eta' \otimes e]$ is in the image of a genuine
  element of $\pi_1(Y^X)$ and $\lVert\eta'-\eta\rVert_{\text{op}}$ is polynomial
  in $\Lip g$.  Then define a homomorphism 
  $\Psi:\mathcal M_Y^* \to \Omega^*(X \times [0,T'])$ by
  \[\Psi(v)=g^*m_Y(v)+dt/T' \wedge (\eta'-\eta),\]
  where $T'$ is sufficiently large that $\Dil(\Psi)=O(L)$.  Then the
  concatenation of $\Phi$ and $\Psi$ lies in the relative homotopy class of a
  genuine homotopy between $f$ and $g$ and has $U$-dilatation $O(L)$.  By the
  improved shadowing principle, we can find such a homotopy with Lipschitz
  constant $O(L)$.
\end{proof}

Now we show that the $L^{4/3}$ bound in the example is sharp:
\begin{thm}
  Let $X$ be any finite $7$-complex and let $f,g:X \to S^3 \vee S^3$ be
  homotopic $L$-Lipschitz maps.  Then there is a homotopy between them through
  $O(L^{4/3})$-Lipschitz maps, with implicit constants depending on $X$.
\end{thm}
\begin{proof}
  We proceed as in the proof of Theorem \ref{thm:2step}.  As before, write 
  $Y=S^3 \vee S^3$ and
  \[\mathcal M_Y^* = (\Lambda(a_1,a_2,b,c_1,c_2,\ldots), da_i=0, db=a_1a_2, dc_i=ba_i,\ldots).\]
  Since $X$ is $7$-dimensional, these are the only generators on which the
  algebraic homotopy we construct will be nonzero.

  We start by applying the construction in Theorem \ref{thm:2step} to the
  subalgebra $\Lambda(a_1,a_2,b)$, using the same notation.  We know that, for
  the homomorphism $\ph$ constructed in the first step,
  $\lVert\ph(b)\rVert_\infty=O(L^6)$.  Since $\tilde\ph$ is constructed so that
  \[[\ph(b)]=\tilde\ph(a_1) \cup g^*a_2+g^*a_1 \cup \tilde\ph(a_2) \in H^n(X;\RR),\]
  and the $g^*a_i$ lie in a lattice in $H^3(X;\RR)$, the forms
  $\alpha\tilde\ph(a_i) \in \Omega^2(X)$ also satisfy
  $\lVert\alpha\tilde\ph(a_i)\rVert_{\text{op}}=O(L^6)$; moreover so does
  $\int_0^1 \Phi_2(db)$.  This allows us to choose $T=L^3$ to obtain
  $U$-dilatation $O(L)$ for the homotopy $\Phi$.

  Now we extend $\Phi$ over $c_1$ and $c_2$ to create a homotopy $\tilde\Phi$
  between $f^*m_Y$ and a map $\psi$ which coincides with $g^*m_Y$ on
  $\Lambda(a_1,a_2,b)$.  By the argument of Prop.~\ref{prop:extHtpy}, we can
  choose this so that
  \[\lVert\tilde\Phi(c_i)\rVert_{\text{op}} \leq C\lVert\tilde\Phi(a_i)\rVert_{\text{op}} \cdot \lVert\tilde\Phi(b)\rVert_{\text{op}} = O(L^{12}).\]

  Next, we concatenate $\tilde\Phi$ with a second homotopy constructed
  similarly to $\Phi_2$ in the proof of Theorem \ref{thm:2step}.  Namely, we
  first consider the obstruction to extending the homotopy which is constant on
  $\Lambda(a_1,a_2,b)$ to the $c_i$.  This obstruction is in the image of a
  homomorphism
  \[[\Lambda(a_1,a_2,b),\Omega^*X \otimes \Lambda(e^{(1)})]_{g^*m_Y} \to \Hom(\langle c_1,c_2\rangle,H^*(X;\RR))\]
  sending $[g^*m_Y+\eta \otimes e] \mapsto [\eta d]$.  Like before, choose a
  representative $g^*m_Y+\tilde\psi \otimes e$ that maps to $[\psi d]$.  Then
  we define $\Phi_3:\mathcal M_Y^* \to \Omega^*X \times [0,T'']$ by
  \begin{align*}
    \Phi_3(a_i) &= g^*m_Y(a_i)-dt/T'' \wedge \alpha\tilde \psi(a_i) && i=1,2 \\
    \Phi_3(b) &= -dt/T \wedge \alpha\tilde \psi(b) \\
    \Phi_3(c_i) &= (1-t/T)\psi(c_i)-dt/T'' \wedge \sigma(c_i)_i && i=1,2,
  \end{align*}
  where $d\sigma_i=\psi(c_i)+\int_0^{T''} \Phi_3(dc_i)$.  For $T''$
  sufficiently large, $\lVert\Phi_3(c_i)\rVert_{\text{op}}=O(L^{12})$ and
  therefore the $U$-dilatation of $\Phi_3$ is $O(L^{4/3})$.

  Finally, we proceed as in the previous proof to construct a self-homotopy
  $\Psi$ of $g^*m_Y$ such that the concatenation of $\tilde\Phi$, $\Phi_3$ and
  $\Psi$ is in the relative homotopy class of a geometric homotopy between $f$
  and $g$; $\Psi$ can be chosen to have dilatation $O(L)$ if one makes the time
  interval sufficiently large.  By the improved shadowing principle, we can
  then find a geometric homotopy with Lipschitz constant $O(L^{4/3})$.
\end{proof}

\bibliographystyle{amsalpha}
\bibliography{liphom}

\end{document}